\documentclass[12pt,a4paper, reqno]{amsart}
\usepackage{amsmath,verbatim,amsthm,amssymb,accents,tensor,slashed}
\usepackage{tikz-cd}
\usepackage[all]{xy}
\usepackage{comment}
\usepackage{fullpage}
\newtheorem{theorem}{Theorem}[section]
\newtheorem{proposition}[theorem]{Proposition}

\newtheorem{corollary}[theorem]{Corollary} 
 
\newtheorem{lemma}[theorem]{Lemma} 

\theoremstyle{definition} 
\numberwithin{equation}{section}
\renewcommand\({\Big(} 
\renewcommand\){\Big)} 
\newcommand\<{\langle} 
\renewcommand\>{\rangle}

\newcommand\er{\eqref}
\newcommand\be[1]{\begin{equation}\label{#1}}
\newcommand\ee{\end{equation}} 
\newcommand\bm{\begin{pmatrix}}
\renewcommand\em{\end{pmatrix}}
\newcommand\bd{\begin{vmatrix}}
\newcommand\ed{\end{vmatrix}}
\renewcommand\i[1]{\item{#1}}

\newcommand\I{\int\limits}
\renewcommand\S{\sum\limits}
\renewcommand\P{\prod\limits}
\newcommand\U{\bigcup\limits}
\newcommand\C{\bigcap\limits}

\newcommand\dl{\partial}

\newcommand\op{\oplus}

\newcommand\oc{\circ}

\newcommand\oo{\infty}
\newcommand\iu{\cup}
\newcommand\ui{\cap}
\newcommand\ic{\subset}
\newcommand\ci{\supset}

\newcommand\xx{\times}

\newcommand\xt{\otimes}

\newcommand\ap{\approx}
\newcommand\qi{\mapsto}

\renewcommand\le{\leqslant}
\renewcommand\ge{\geqslant}

\newcommand\sm{\setminus}

\renewcommand\.{\cdot}

\renewcommand\;{\ldots}
\newcommand\la{\alpha}
\newcommand\lb{\beta}

\newcommand\lc{\chi}
\renewcommand\lg{\gamma}
\newcommand\lG{\Gamma}
\newcommand\ld{\delta}
\newcommand\lD{\Delta}
\newcommand\Le{\varepsilon}
\newcommand\lh{\eta}

\newcommand\lk{\varkappa}
\renewcommand\ll{\lambda}

\newcommand\lm{\mu}
\renewcommand{\ln}{\nu}

\newcommand\lf{\phi}

\newcommand\lF{\Phi}
\renewcommand\lq{\psi}
\newcommand\lQ{\Psi}
\newcommand\lp{\pi}

\newcommand\lS{\Sigma}

\newcommand\Lt{\tau}
\newcommand\lo{\omega}
\newcommand\lO{\Omega}
\newcommand\lx{\xi}

\newcommand\lz{\zeta}
\newcommand\AL{\mathcal A}

\newcommand\DL{\mathcal D} 
\newcommand\EL{\mathcal E}

\newcommand\KL{\mathcal K} 
 
\newcommand\ML{\mathcal M}

\newcommand\PL{\mathcal P}

\newcommand\VL{\mathcal V}

\newcommand\Cl{{\mathbf C}}

\newcommand\Nl{{\mathbf N}}

\newcommand\Tl{{\mathbf T}}

\newcommand\hL{\mathfrak h}
\newcommand\kL{\mathfrak k}

\newcommand\bb[4]{\bm{#1}&{#2}\\{#3}&{#4}\em}

\newcommand\q{\quad} 
\renewcommand\l{\ell}

\renewcommand\b{{\mathrel{\scalebox{0.8}{$\Box$}}}} 
\renewcommand\c[1]{{\check{#1}}}

\newcommand\f[2]{{\frac{#1}{#2}}}
 
\newcommand\h[1]{{\widehat{#1}}} 
\renewcommand\o[1]{{\overline{#1}}}

\newcommand\x[1]{{\mathrm{#1}}} 
\newcommand\y[1]{{\undertilde{#1}}}

\DeclareFontFamily{U}{mathb}{\hyphenchar\font45}
\DeclareFontShape{U}{mathb}{m}{n}{
 <5><6><7><8><9><10> gen * mathb
 <10.95> mathb10
}{}
\DeclareSymbolFont{mathb}{U}{mathb}{m}{n}
\DeclareMathSymbol{\VDash}3{mathb}{"28}

\begin{document}
\setcounter{section}{-1}

\title[Hilbert Modules on Bounded Symmetric Domains]{$K$-invariant Hilbert Modules and Singular Vector Bundles on Bounded Symmetric Domains}
\author{Harald Upmeier}
\medskip
\address{Fachbereich Mathematik, Universit\"at Marburg, D-35032 Marburg, Germany}
\email{upmeier{@}mathematik.uni-marburg.de}

\dedicatory{Dedicated to the Memory of Ottmar Loos}



\begin{abstract} We show that the "eigenbundle" (localization bundle) of certain Hilbert modules over bounded symmetric domains of rank $r$ is a "singular" vector bundle (linearly fibrered complex analytic space) which decomposes as a stratified sum of homogeneous vector bundles along a canonical stratification of length $r+1$. The fibres are realized in terms of representation theory on the normal space of the strata.
\end{abstract}

\maketitle

\section{Introduction}
Let $\PL_E\ap\Cl[z_1,\;,z_d]$ be the algebra of all polynomials on a vector space $E\ap\Cl^d.$ Let $\ML_\lz$ be the maximal ideal at $\lz\in E.$ For any ideal $I\ic\PL_E$ the quotient module
$$\y I_\lz:=I/\ML_\lz I$$
has finite dimension, and the disjoint union
$$\y I:=\U_{\lz\in E}\y I_\lz$$
has the structure of a "linearly fibrered complex analytic space" \cite{F}, also called a "singular vector bundle." We call 
$\y I$ the "localization bundle" of the ideal $I.$ For example, if $I$ is a prime ideal whose vanishing locus $X$ consists only of smooth points, then a result of Duan-Guo \cite{DG} states that the localization bundle has rank 1 on the regular set $E\sm X$ whereas on $X$ the bundle is isomorphic to the (dual) normal bundle to the submanifold $X.$ Thus we have a "stratification" of length 2. In this paper we consider ideals over bounded symmetric domains of arbitrary rank $r$ and obtain localization bundles which are stratified of length $r+1.$  

Besides their interest in complex and algebraic geometry, as the dual objects of coherent analytic module sheaves 
\cite{BMP,F}, these bundles play a fundamental role in multi-variable operator theory for commuting tuples $(T_1,\;,T_d)$ of non-selfadjoint operators acting on a Hilbert space $H$ of holomorphic functions on a bounded domain $D\ic\Cl^d.$ Here a central concept is the so-called "eigenbundle" 
$$\y H_\lz:=\{\lf\in H:\ T_j^*\lf=\o{\lz_j}\.\lf\ \forall\ 1\le j\le d\}$$
of $H$ viewed as a Hilbert module \cite{CD}. The main idea is that the differential-geometric properties (Chern connection, curvature, etc.) of the eigenbundle $\y H,$ viewed as a (singular) hermitian holomorphic vector bundle over $D,$ should characterize the underlying operator tuple up to unitary equivalence. In this sense the complex-analytic properties of $\y H$ are analogous to the spectral theorem in the self-adjoint case. If $H=\o I$ is the Hilbert closure of a polynomial ideal $I$ then there is a natural isomorphism $\y H\ap\y I|_D.$

In the original approach by Cowen-Douglas \cite{CD} this program was fully realized for a certain class of operator tuples where the eigenbundle is a genuine vector bundle without singularities. In general, for example in the situation of the Duan-Guo theorem, the eigenbundle is not a smooth vector bundle anymore and its fibre dimension varies over different strata within the domain $D.$ If $D=G/K$ is a bounded symmetric domain  of arbitrary rank $r$ there is a canonical stratification into "Kepler varieties" defined by a rank condition, and our main result characterizes the localization bundle $\y J^\ll$ for certain polynomial ideals $J^\ll$ determined by any partition $\ll$  of length $\le r.$ Only the "fundamental" partitions $\ll=(1,\;,1,0,\;,0)$ give rise to prime ideals. 

We show that the associated eigenbundle is stratified of length $r+1$ (including the open stratum) and the fibre at a point $\lz$ is described in terms of certain polynomials (instead of linear forms as for prime ideals and smooth points) on the normal space at $\lz.$ The realization is given by an explicit geometric construction, taking a projection of a polynomial to the normal space, and the main challenge is to identify the kernel of this projection map. The theory of Jordan algebras and Jordan triples \cite{C,FK2} is used to carry out the general discussion in a uniform way without using the classification of bounded symmetric domains.  

The results of this paper have numerous consequences, further developed in \cite{U3}. For example, for any Hilbert closure $H=\o J^\ll,$ with reproducing kernel $\KL(z,\lz),$ the localization bundle is explicitly identified with the subbundle 
$\y H\ic D\xx H,$ by taking certain "normal derivatives" of kernel functions $\KL^s$ constructed from $\KL.$ This is important for introducing the hermitian metric on the singular vector bundle. Moreover, in the spirit of the Borel-Weil-Bott theorem, the fibres $\y J_\lz^\ll$ can be described by holomorphic sections of line bundels over flag manifolds. It is also shown how to extend the analysis to arbitrary $K$-invariant ideals, in particular the (Jordan) determinantal ideals which are defined by vanishing conditions on the underlying stratification. 
   
\section{Hilbert modules and their eigenbundles}
Let $D$ be a bounded domain in a finite dimensional complex vector space $E\ap\Cl^d.$ Denote by $\PL_E\ap\Cl[z_1,\;,z_d]$ the algebra of all polynomials on $E.$ A Hilbert space $H$ of holomorphic functions $f$ on $D$ (supposed to be scalar-valued) is called a {\bf Hilbert module} if for any polynomial $p\in\PL_E$ the multiplication operator $T_pf:=pf$ leaves $H$ invariant and is bounded. Using the adjoint operators 
$T_p^*,$ the closed linear subspace
$$\y H_\lz:=\{f\in H:\ T_p^*f=\o{p(\lz)}f\ \forall\ p\in\PL_E\}$$
is called the {\bf joint eigenspace} at $\lz\in D.$ Since $T_pT_q=T_{pq}$ for polynomials $p,q$ it suffices to consider linear functionals or just the coordinate functions. The disjoint union
$$\y H=\U_{\lz\in D}\ \y H_\lz$$
becomes a subbundle of the trivial vector bundle $D\xx H,$ which is called the {\bf eigenbundle} of $H,$ although it is not locally trivial in general. One also requires that the fibres have finite (non-constant) dimension and their union is total in $H.$ For more details, cf. \cite{KM,BMP}.

The eigenbundle is closely related to the concept of {\bf localization}. For a given point $\lz\in D,$ the complex numbers become a $\PL_E$-module, denoted by $\Cl_\lz,$ under the action $p\.\lx:=p(\lz)\lx.$ Denote by
\be{3}\ML_\lz=\{f\in\PL_E:\ f(\lz)=0\}\ic\PL_E\ee 
the maximal ideal at $\lz\in E.$ Define the module tensor product 
\be{1}H\xt_{\PL_E}\Cl_\lz= H/\o{\ML_\lz H}\ee
where $\o{\ML_\lz H}$ is the closed submodule generated by $T_p\lq-p(\lz)\lq,$ for all $p\in\PL_E$ and $\lq\in H.$

\begin{lemma}\label{a} The map $\lf\qi[\lf]$ from $\y H_\lz$ to $H/\o{\ML_\lz H}$ is a Hilbert space isomorphism, with inverse given by
$$H/\o{\ML_\lz H}\to\y H_\lz,\q f+\o{\ML_\lz H}\qi\lp_\lz f,$$
where $\lp_\lz: H\to\y H_\lz$ is the orthogonal projection. Thus $\y H_\lz$ is the "quotient module" for the submodule 
$\o{\ML_\lz H}.$ 
\end{lemma}
\begin{proof} Let $z_i,\ 1\le i\le d$ denote the coordinate functions. By definition we have
$$\y H_\lz=\C_{i=1}^d \ker T_{z_i-\lz_i}^*=\C_{i=1}^d(T_{z_i-\lz_i}H)^\perp
=(\S_{i=1}^d T_{z_i-\lz_i}H)^\perp=(\ML_\lz H)^\perp\ap H/\o{\ML_\lz H}.$$
\end{proof}

Classical examples of Hilbert modules are the {\bf Bergman space} $H^2(D)$ of square-integrable holomorphic functions, whose reproducing kernel is called the Bergman kernel, and the {\bf Hardy space} $H^2(\dl D)$ if $D$ has a smooth boundary $\dl D.$ For general Hilbert modules $H$, a {\bf reproducing kernel function} is a sesqui-holomorphic function $\KL(z,\lz)$ on $D\xx D$ such that for each $\lz\in D$ the holomorphic function 
$$\KL_\lz(z):=\KL(z,\lz)$$ 
belongs to $H,$ and we have
$$\lq(z)=(\KL_z|\lq)_H$$ 
for all $\lq\in H$ and $z\in D.$ Here $(\lf|\lq)_H$ is the inner product on $H$ (anti-linear in the first variable). Thus 
$H$ is the closed linear span of the holomorphic functions $\KL_\lz,$ where 
$\lz\in D$ is arbitrary. In terms of an orthonormal basis $\lf_\la$ of $H,$ we have
$$\KL(z,\lz)=\S_\la\ \lf_\la(z)\o{\lf_\la(\lz)}.$$ 
The identity 
$$(T_p^*\KL_\lz|\lq)_H=(\KL_\lz|p\lq)_H=p(\lz)\lq(\lz)=p(\lz)\ (\KL_\lz|\lq)_H=(\o{p(\lz)}\ \KL_\lz|\lq)_H$$
for $p\in\PL_E$ and $\lq\in H$ shows that $T_p^*\KL_\lz=\o{p(\lz)}\ \KL_\lz,$ so that
$$\KL_\lz\in\y H_\lz$$ 
for each $\lz\in D.$ If $\KL$ has no zeros (e.g., the Bergman kernel of a strongly pseudo-convex domain) then the eigenbundle 
$\y H$ is spanned by the reproducing kernel functions $\KL_\lz$ and hence becomes a {\bf hermitian holomorphic line bundle}.

In general, Hilbert modules have reproducing kernel function which vanish along certain analytic subvarieties of $D.$ In this case the eigenbundle is not locally trivial, its fibre dimension can jump along the varieties and we obtain a {\bf singular vector bundle} on $D,$ also called a "linearly fibered complex analytic space". Such singular vector bundles are important in Several Complex Variables since, by \cite{F}, they are in duality with the category of {\bf coherent analytic module sheaves}, whereas (regular) vector bundles correspond to locally free sheaves.

An important class of Hilbert modules is given by the Hilbert closure $H=\o I$ of a polynomial ideal $I\ic\PL_E.$ In this case we define, analogous to \er{1}, the {\bf localization}
$$\y I_\lz:=I/\ML_\lz I$$ 
at $\lz\in E,$ and call the disjoint union 
$$\y I:=\U_{\lz\in E}\y I_\lz$$
the {\bf localization bundle} over $E.$ We first show that this has finite rank.

\begin{proposition} Let $p_1,\;,p_t$ be generators of $I.$ Then for any $\lz\in E$ the linear map
$$\Cl^t\to\y I_\lz,\ (a_1,\;,a_t)\qi\ML_\lz I+\S_{i=1}^t a_i\ p_i$$
is surjective, and hence $\dim I/\ML_\lz I\le t.$
\end{proposition}
\begin{proof} Write $f\in I$ as $f=\S_{i=1}^t h_i\ p_i$ with $h_i\in\PL_E.$ Then
$$f-\S_i h_i(\lz)p_i=\S_i(h_i-h_i(\lz))p_i\in\ML_\lz I.$$ 
Putting $a_i=h_i(\lz),$ the assertion follows.
\end{proof}

For any ideal $I\ic\PL_E$ consider the vanishing locus
$$\VL^I=\{\lz\in E:\ p(\lz)=0\ \forall\ p\in I\}$$
and the "regular set"
$$\c E=E\sm\VL^I.$$

\begin{proposition}\label{b} For $\lz\in\c E$ there is an isomorphism
\be{2}\y I_\lz\to\Cl,\q f+\ML_\lz I\mapsto f(\lz).\ee
Thus the localization bundle has rank 1 on the regular set $E\sm\VL^I.$
\end{proposition}
\begin{proof} If $f=gh,$ with $g\in\ML_\lz$ and $h\in I,$ then $g(\lz)=0$ and hence $f(\lz)=g(\lz)h(\lz)=0.$ Thus the map 
\er{2} is well-defined. Since $\lz\in\c E$ there exists $p\in I$ such that $p(\lz)\ne 0.$ Thus the map \er{2} is non-zero and hence surjective. To show injectivity, we may assume $p(\lz)=1.$ If $f\in I$ satisfies $f(\lz)=0,$ then
$$f=(1-p)f+fp\in\ML_\lz I$$
since both $1-p$ and $f$ belong to $\ML_\lz.$
\end{proof}

Passing to a Hilbert module completion $H=\o I$ it is shown in \cite{DG} that $\dim I/\ML_\lz I=\dim H/\o{\ML_\lz H}<+\oo$ and hence the map
$$\y I_\lz\to\y H_\lz,\q f+\ML_\lz I\qi\lp_\lz f$$
is an isomorphism. The difference between $\y I$ and $\y H$ is that $\y H$ carries an additional hermitian fibre metric, being embedded in $D\xx H.$ Also, $\y I$ is defined on all of $E,$ whereas $\y H$ is defined only on $D.$ In this sense we have 
\be{13}\y H=\y I|_D,\ee
but equipped with a holomorphic hermitian metric. As shown in \cite{BMP} the analytic module sheaf associated with $H$ is coherent, so that these bundles become "complex-analytic linear fibre spaces" in the sense of \cite{F}. We prefer the term 
{\bf singular vector bundle.} 

Let us recall the following \cite[Lemma 2.3]{BMP}:

\begin{lemma}\label{c} Let $p_1,\;,p_t$ be a finite set of generators of $I.$ Then $f\in\y H_\lz$ satisfies
$$\o{p_i(\lz)}f=(p_i|f)_H\ \KL_\lz$$
for all $i.$ 
\end{lemma}
\begin{proof} For any $1\le i,j\le t$ we have
$$p_i(\lz)(f|p_j)_H=(\o{p_i(\lz)}f|p_j)_H=(T_{p_i}^*f|p_j)_H=(f|p_ip_j)_H=p_j(\lz)(f|p_i)_H.$$
Let $q_j\in\PL_E$ be arbitrary. Using the reproducing property it follows that
$$p_i(\lz)\(f|\S_j p_jq_j\)_H=p_i(\lz)\S_j(T_{q_j}^*f|p_j)_H=p_i(\lz)\S_jq_j(\lz)(f|p_j)_H$$
$$=(f|p_i)_H\S_jp_j(\lz)q_j(\lz)=(f|p_i)_H\S_j(\KL_\lz|p_jq_j)_H=\((p_i|f)_H\ \KL_\lz|\S_jp_jq_j\)_H.$$
Since the vector subspace $\{\S_j p_jq_j:\ q_j\in\PL_E\}$ is dense in $H=\o I,$ the assertion follows.
\end{proof}

Define the "regular set"
$$\c D:=D\sm\VL^I=\U_{j=1}^t\{\lz\in D:\ p_j(\lz)\ne 0\}$$
as a dense open subset of $D.$

\begin{corollary}\label{d} For $\lz\in\c D:=D\ui\c E=D\sm\VL^I$ we have
\be{4}\y H_\lz=\Cl\KL_\lz.\ee
Thus the eigenbundle $\y H$ restricted to the regular set is a holomorphic line bundle spanned by the reproducing kernel functions $\KL_\lz,\ \lz\in\c D.$ 
\end{corollary}
\begin{proof} Lemma \ref{c} implies \er{4} since $p_i(\lz)\ne 0$ for some $i.$
\end{proof}

The behavior of $\y H$ on the singular set $D\sm\c D$ is more complicated and has so far been studied mostly when the vanishing locus of the reproducing kernel is a smooth subvariety of $D,$ for example given as a complete intersection of a regular sequence of polynomials. The case where $I$ is a {\bf prime ideal} whose vanishing locus $X:=\VL^I$ consists of 
{\bf smooth points} has been studied by Duan-Guo \cite{DG}. They showed that for $\lz\in D\sm X$
$$\y H_\lz=\<\KL_\lz\>$$
is 1-dimensional, whereas $\y H|_X$ is isomorphic to the (dual) {\bf normal bundle} of the submanifold $X.$ Thus we have a 
{\bf stratification of length 2}. We consider a more general situation for bounded symmetric domains $D$ of arbitrary rank $r,$ where we have a stratification of length $r+1$. Here the relevant algebraic varieties are not smooth and the ideal $I$ is not prime in general. 

\section{Bounded symmetric domains and Jordan triples}
We use the Jordan theoretic description of bounded symmetric domains \cite{Ar,C,FK2,L,U1}. Every (hermitian) bounded symmetric domain can be realized as the unit ball, with respect to the so-called spectral norm, in a complex vector space $E$ endowed with a {\bf Jordan triple product}. This is a ternary operation
$$E\xx E\xx E\to E\quad (u,v,w)\mapsto\{uv^*w\}$$
which is symmetric bilinear in the outer variables $(u,w),$ conjugate-linear in the inner variable $v$ and satisfies the "Jordan triple identity"
\be{6}[u\b v^*,z\b w^*]=\{uv^*z\}\b w^*-z\b\{wu^*v\}^*\ee
for all $u,v,z,w\in E.$ Here 
\be{5}(u\b v^*)z:=\{uv^*z\}\ee 
denotes the "triple multiplication" operator. The "star" occuring here is a formal symbol. A complex vector space $E$ carrying such a structure will be called a {\bf hermitian Jordan triple} (or $J^*$-triple) if the sesqui-linear product
$$(u|v)_0:=\x{tr}\ u\b v^*$$
is hermitian and positive-definite. Geometrically, this inner product coincides with the Bergman metric at the origin 
$0\in D.$ We consider only finite dimensional Jordan triples $E,$ although the Jordan theoretic description of symmetric manifolds carries over to the case of Banach manifolds \cite{U0,C}. We also assume that $E$ is irreducible of rank $r.$

The primary example of a hermitian Jordan triple is the {\bf matrix space} $E=\Cl^{r\xx s},$ with $r\le s,$ endowed with the anti-commutator triple product
$$\{uv^*w\}=uv^*w+wv^*u.$$
The associated domain is the matrix unit ball for the operator norm. We will often illustrate the general theory with this example. If $r=s$ one can take $v=e$ (unit matrix) and obtains the classical anti-commutator
$$\{ue^*w\}=uw+wu$$
which is the prototype of the so-called Jordan algebras \cite{FK2}. For rank $r=1,$ we obtain the unit ball 
$D\ic E=\Cl^{1\xx d}$ with Jordan triple product $\{xy^*z\}=(x|y)z+(z|y)x.$ The unit disk $D\ic E=\Cl$ corresponds to the  Jordan triple product $\{xy^*z\}=2x\o y z.$

Irreducible hermitian Jordan triples of rank $r$ are classified by two {\bf characteristic multiplicities} $a$ and $b$ such that
$$\f dr=1+\f a2(r-1)+b.$$
If $b=0$ the domain $D$ is called of {\bf tube type}. In this case $E$ is a {\bf Jordan algebra}. The full classification is
\begin{itemize}
\i{}{\bf matrix triple} $E=\Cl^{r\xx s},\ \x{rank}=r\le s,\ a=2,\ b=s-r$ (complex case)
\i{}{\bf symmetric matrices} $a=1,\ b=0$ (real case)
\i{}{\bf anti-symmetric matrices} $a=4$ (quaternion case)
\i{}{\bf spin factor} $E=\Cl^d,\ r=2,\ a=d-2,\ b=0$
\i{}{\bf exceptional Jordan triples} of dimension 27 ($r=3,\ a=8,\ b=0$) and 16 ($r=2$) (octonion case)
\end{itemize}

By the classification the only Jordan triple of rank 1 is the row-space $E=\Cl^s=\Cl^{1\xx s}.$ Its unit ball 
$D\ic\Cl^s=\Cl^{1\xx s}$ is the only bounded symmetric domain which is strictly pseudo-convex or has a smooth boundary.

Let $G$ be the identity component of the biholomorphic automorphism group of $D.$ Then $D=G/K,$ where the stabilizer subgroup 
$K\ic G$ at the origin consists of linear Jordan triple automorpisms of $E.$ The "structure group" $\h K,$ a complexification of $K,$ is a complex Lie subgroup of $\x{GL}(E)$ endowed with an involution $h\qi h^*$ such that
$$h(u\b v^*)h^{-1}=(hu)\b(h^{-*}v)^*$$
for all $u,v\in E.$ Let $E\b E^*$ denote the vector space spanned by linear transformations $u\b v^*,$ defined in \er{5}, for $u,v\in E$. By the Jordan triple identity \er{6} this is a Lie algebra which coincides with the Lie algebra $\h\kL$ of $\h K.$ We have 
$$\x{tr}\ u\b v^*=\x{const}\ (u|v)$$ 
for the $K$-invariant inner product on $E$ and put
$$E\b_0 E^*:=\{A\in E\b E^*:\ \x{tr}\ A=0\}.$$
In the matrix case $E=\Cl^{r\xx s}$ $\h K$ consists of all transformations of the form
$$hz=azd$$
where $a\in\x{GL}_r(\Cl),\ d\in\x{GL}_s(\Cl).$ The adjoint is $h^*z=a^*zd^*.$ Also, $\h\kL=E\b E^*$ consists of all transformations
$$z\qi az+zd$$
where $a\in\Cl^{r\xx r}$ and $d\in\Cl^{s\xx s}.$ 

In a Jordan triple $E$ the central notion of tripotent generalizes the idempotents in an algebra. An element $c\in E$ is called a {\bf tripotent} (triple idempotent) if the Jordan triple product satisfies 
$$\{cc^*c\}=2c.$$ 
For $0\le j\le r$ the set $S_j$ of all tripotents of rank $j$ is a compact real-analytic manifold which is homogeneous under 
$K.$ We have $S_0=\{0\},$ the elements of $S_1$ are called {\bf minimal} tripotents, and the maximal tripotents $S:=S_r$ form the 
{\bf Shilov boundary} of $D$ \cite{L}.

For the matrix triple $\Cl^{r\xx s}$ the tripotents are the so-called partial isometries characterized by $cc^*c=c.$ The minimal tripotents are the rank 1 matrices $\lx\lh^*,$ where $\lx\in\Cl^r$ and $\lh\in\Cl^s$ are unit vectors. For $r=s,$ any self-adjoint projection ($c=c^*=c^2$) and any unitary matrix ($cc^*=c^*c=1$) is a tripotent. In the rank 1 case $E=\Cl^d$ $S=S_1$ consists of all unit vectors. 

A {\bf frame} of $E$ is a family $(e_1,\;,e_r)$ of minimal orthogonal tripotents. Every $\lz\in E$ has a {\bf spectral decomposition}
\be{7}\lz=\S_j\lz_j e_j\ee
where $e_j$ is a frame and $\lz_1\ge\lz_2\ge\;\ge\lz_r\ge 0$ are the {\bf singular values}. For matrices this is the classical singular value decomposition under $U(r)\xx U(s).$

Any tripotent $c$ induces a {\bf Peirce decomposition}
\be{13}E=E_c^2\op E_c^1\op E_c^0\ee
into eigenspaces $E_c^\la=\{z\in E:\ \{cc^*z\}=\la z\}$ for $\la=0,1,2.$ The Peirce spaces $E_c^\la$ are Jordan subtriples of 
$E.$ In the matrix case the tripotent
$$c=\bb{1_\l}000$$
of rank $\l$ induces the Peirce decomposition
$$\Cl^{r\xx s}=\bb{\Cl^{\l\xx\l}}{\Cl^{\l\xx(s-\l)}}{\Cl^{(r-\l)\xx\l}}{\Cl^{(r-\l)\xx(s-\l)}}=\bb{E^2}{E^1}{E^1}{E^0}.$$
For the spin factor $E=\Cl^{2+a}$ of rank 2 consider the minimal tripotent $c=(1/2,i/2,0,\;,0)$ and put
$\o c=(1/2,-i/2,0,\;,0).$ Then $e=c+\o c=(1,0,\;,0)$ is the unit element and we have
$$E_c^2=\Cl\.c,\ E_c^0=\Cl\.\o c=E_{\o c}^2,\ E_c^1=E_{\o c}^1=\{c,\o c\}^\perp\ap\Cl^a.$$
We often abbreviate $E_c:=E_c^2$ and $E^c:=E_c^0.$

The number of non-zero singular values $\lz_j$ in \er{7} is called the {\bf rank} of $\lz.$ For each $0\le\l\le r$ the {\bf Kepler manifold} \cite{EU}
$$\c E_\l=\{\lz\in E:\ \x{rank}(\lz)=\l\}$$
is a complex-analytic manifold which is a $\h K$-orbit containing the compact submanifold $S_\l$ of all tripotents of rank 
$\l.$ For maximal $\l=r$ the set
$$\c E:=\c E_r$$
is an open dense subset of $E.$ The closure
\be{8}\h E_\l:=\{\lz\in E:\ \x{rank}(\lz)\le\l\}=\U_{i=0}^\l\c E_i\ee
of $\c E_\l$ is called the {\bf Kepler variety}. It is irreducible and normal \cite{EU}. The smooth points of $\h E_\l$ coincide with $\c E_\l.$ For matrices $E=\Cl^{r\xx s}$ we obtain the classical "determinantal variety" of all matrices of rank $\le\l.$ The spin factor $E=\Cl^d$ yields the quadric $\h E_1=\{z\in E:\ z\.z=\S_i z_i^2=0\}.$ The identity \er{8} is typical of a {\bf stratification}
$$E=\U_{\l=0}^r\c E_\l,\q\mbox{(disjoint union)}$$
into $r+1$ complex analytic $\h K$-orbits $\c E_\l.$ The maximal stratum $\c E=\c E_r$ is open and dense, and the only closed stratum is the minimal one $\h E_0=\c E_0=\{0\}.$

\section{$K$-invariant ideals and Hilbert modules}
For a $J^*$-triple $E$ let $\PL_E$ denote the algebra of all (holomorphic) polynomials, endowed with the "Fischer-Fock" inner product $(\lf|\lq)$ (anti-linear in the first variable). The natural action
$$(k\.f)(z):=f(k^{-1}z)$$
of $K$ on functions $f$ on $E$ induces a multiplicity-free {\bf Peter-Weyl decomposition} \cite{S,FK1}
$$\PL_E=\S_{\ll\in\Nl_+^r}\PL_E^\ll.$$
Here $\Nl_+^r$ denotes the set of all {\bf partitions} 
$$\ll=(\ll_1,\ll_2,\;,\ll_r)$$ 
of integers $\ll_1\ge\ll_2\ge\;\ge\ll_r\ge0.$ A partition $\ll$ is often identified with its {\bf Young diagram}
\be{14}[\ll]:=\{(i,j):\ 1\le i\le r,\ 1\le j\le\ll_i\}\ee
viewed as a subset of $\Nl\xx\Nl.$ The polynomials in $\PL_E^\ll$ are homogeneous of degree 
$$|\ll|:=\ll_1+\;+\ll_r.$$
Hence for any $n\in\Nl$ the $n$-homogeneous polynomials $\PL_E^n$ decompose as
$$\PL_E^n=\S_{|\ll|=n}\PL_E^\ll.$$
As special cases $\PL_E^{0,\;,0}=\PL_E^0=\Cl$ (constant functions) and $\PL_E^{1,0,\;,0}=\PL_E^1=E^*$ (linear dual space).

Let $\lf_\la^\ll$ be an orthonormal basis of $\PL_E^\ll.$ The sesqui-polynomial 
\be{9}\EL^\ll(z,\lz)=\S_\la\lf_\la^\ll(z)\o{\lf_\la^\ll(\lz)}\ee 
is called the {\bf Fischer-Fock reproducing kernel} for $\ll.$ For example, $\ll=(1,0,\;,0)$ gives rise to the normalized 
$K$-invariant inner product
$$\EL^{1,0,\;,0}(z,\lz)=(z|\lz).$$

\begin{lemma}\label{e} Let
$$\lc_\ll(k):=\x{tr}_{\PL_E^\ll}(k\.)$$
denote the character of the representation $\PL_E^\ll$ and put
$$d_\ll:=\dim\PL_E^\ll.$$ 
Then the orthogonal projection 
$$\lp^\ll:\PL_E\to\PL_E^\ll,\ f\qi\lp^\ll f=:f^\ll,$$ 
is given by the "character integral formula"
\be{10}\lp^\ll f=d_\ll\I_K dk\ \lc_\ll(k)\ f\oc k\ee
for all $f\in\PL_E,$ 
\end{lemma}
\begin{proof} By definition,
$$\lc_\ll(k)=\S_\lg(\lf_\lg^\ll|\lf_\lg^\ll\oc k^{-1})=\S_\lg(\lf_\lg^\ll\oc k|\lf_\lg^\ll).$$
For any partition $\lm$ it follows from Schur orthogonality that
$$d_\ll\I_K dk\ \lc_\ll(k)\ (\lf_\la^\lm|\lf_\lb^\lm\oc k)
=d_\ll\I_K dk\ \S_\lg(\lf_\lg^\ll\oc k|\lf_\lg^\ll)\ (\lf_\la^\lm|\lf_\lb^\lm\oc k)$$
$$=\ld_\ll^\lm\S_\lg(\lf_\lg^\ll|\lf_\lb^\lm)(\lf_\la^\lm|\lf_\lg^\ll)
=\ld_\ll^\lm\ld_\la^\lb(\lf_\la^\ll|\lf_\lb^\lm)=(\lf_\la^\ll|\lp^\ll\lf_\lb^\lm).$$
since $\PL_E^\ll$ and $\PL_E^\lm$ are inequivalent $K$-modules if $\ll\ne\lm.$
\end{proof}

Any $K$-invariant Hilbert module $H$ of holomorphic functions on $D$ carries a $K$-invariant inner product $(\lf|\lq)_H$ which is uniquely determined by the condition
$$(p|q)=a_\ll\ (p|q)_H$$
for all $p,q\in\PL_E^\ll\ic H$ where $a_\ll>0$ are constants, and $\ll$ runs over all partitions such that 
$\PL_E^\ll\ic H.$ Taking only those partitions we thus have a Fourier type decomposition
$$H=\S_\ll\PL_E^\ll\q\mbox{(Hilbert sum)}$$
determined by the sequence $(a_\ll)$ of coefficients. Using \er{9} it follows that $H$ has the reproducing kernel
$$\KL(z,\lz)=\S_\ll a_\ll\ \EL^\ll(z,\lz),$$
summed over all partitions $\ll$ such that $\PL_E^\ll\ic H.$ 

\begin{lemma}\label{f} Let $H$ be a $K$-invariant Hilbert module. If $f\in\y H_0$ then $f^\lm\in\y H_0$ for all $\lm.$
\end{lemma}
\begin{proof} Using the character integration formula \er{10}, this follows from the fact that $\y H_0$ is a closed subspace of 
$H$ which is $K$-invariant.
\end{proof}

A $K$-invariant reproducing kernel function $\KL(z,\lz)=\KL_\lz(z)$ satisfies $\KL(kz,k\lz)=\KL(z,\lz)$ for all $k\in K.$ By analytic continuation this implies
$$\KL_\lz\oc h^*=\KL_{h\lz}$$
for all $\lz\in D$ and $h\in\h K$ such that $h\lz\in D.$

\begin{lemma}\label{g} Let $H$ be a $K$-invariant Hilbert module and $h\in\h K.$ Then
$$\y H_\lz\oc h^*=\y H_{h\lz}$$
\end{lemma}
\begin{proof} For $f,g\in H$ we have $(f\oc k^{-1}|g)_H=(f|g\oc k)_H$ for all $k\in K$ since $H$ is $K$-invariant. It follows that
\be{11}(f\oc h^*|g)_H=(f|g\oc h)_H\ee
for all $h\in\h K,$ since both sides of \er{11} are holomorphic in $h$ and \er{11} holds for $h\in K$ where 
$h^*=h^{-1}.$ Now let $g=\lf\lq$ with $\lf\in\ML_{h\lz}$ and $\lq\in H.$ Then 
$$g\oc h=(\lf\oc h)(\lq\oc h)$$
with $\lq\oc h\in H$ and $\lf\oc h\in\PL_E$ satisfies $(\lf\oc h)(\lz)=\lf(h\lz)=0.$ Therefore $\lf\oc h\in\ML_\lz$ and
$g\in\ML_\lz H.$ If $f\in\y H_\lz$ then $(f\oc h^*|g)=(f|g\oc h)=0$ by Lemma \ref{a}. Hence \er{11} implies 
$f\oc h^*\in\y H_{h\lz}.$ Therefore $\y H_\lz\oc h\ic\y H_{h\lz}.$ Passing to $h^{-1}$ yields equality.
\end{proof}

The structure group $\h K$ acts transitively on each Kepler manifold $\c E_\l.$ It follows that
$$\c E_\l=\h K/\h K^c$$    
where $\h K^c=\{\lg\in\h K:\ \lg c=c\}$ for some (fixed) tripotent $c\in S_\l.$ If $\lg\in\h K^c$ then $\y H_c\oc\lg^*=\y H_c$ by Lemma \ref{g} and this defines an action of $\h K^c$ on $\y H_c.$ As usual define the {\bf homogeneous vector bundle}
$$\h K\xx_{\h K^c}\y H_c:=\{[h,\lf]=[h\lg,\lf\oc\lg^{-*}]:\ h\in\h K,\ \lg\in\h K^c,\ \lf\in\y H_c\},$$
endowed with the $\h K$-action $h\.[h',\lf]:=[hh',\lf].$

\begin{proposition}\label{h} For a $K$-invariant Hilbert module $H$ the restriction $\y H|_{\c E_\l}$ of the eigenbundle 
$\y H$ to each stratum $\c E_\l$ is $\h K$-isomorphic to the homogeneous vector bundle
$$\y H|_{\c E_\l}\ap \h K\xx_{\h K^c}\y H_c.$$
For the fibre at $\lz\in\c E_\l$ the isomorphism is given by $\Le_\lz f=[h,f\oc h^{-*}]$ for $f\in\y H_\lz,$ where $h\in\h K$ satisfies $hc=\lz.$
\end{proposition}
\begin{proof} If $f\in\y H_\lz,$ with $\lz=hc,$ then $f\oc h^{-*}\in\y H_c$ by Lemma \ref{g} giving the equivalence class 
$\Le_\lz f=[h,f\oc h^{-*}].$ For any $\lg\in\h K^c$ the definition \er{12} yields 
$[h\lg,f\oc(h\lg)^{-*}]=[h\lg,f\oc h^{-*}\oc\lg^{-*}]=[h,f\oc h^{-*}],$ showing that $\Le_\lz$ is well-defined. To show equivariance let $h'\in\h K.$ Then $f\oc {h'}^*\in\y H_{h'\lz}$ by Lemma \ref{g}. Since $h'\lz=h'hc$ we have
$$\Le_{h'\lz}(f\oc{h'}^*)=[h'h,(f\oc{h'}^*)\oc(h'h)^{-*}]=[h'h,f\oc h^{-*}]=h'[h,f\oc h^{-*}]=h'\.\Le_\lz f.$$
\end{proof}

Thus for $K$-invariand Hilbert modules $H$ it suffices to determine the fibre $\y H_c$ for a fixed tripotent $c\in S_\l.$ The full eigenbundle will then automatically be a "stratified union" of homogeneous vector bundles along each stratum $\c E_\l.$ Note that these vector bundles are really anti-holomorphic in the fibre variable $\lz$, as is standard for Hilbert modules of holomorphic functions.

The basic {\bf examples} of Hilbert modules on a bounded symmetric domain are the so-called weighted Bergman spaces and Hardy type spaces on the Shilov boundary and the other boundary strata \cite{AU}. Using the spectral decomposition \er{7}, there exists a unique $K$-invariant sesqui-polynomial $\lD:E\xx E\to\Cl$ such that
$$\lD(\lz,\lz)=\S_{j=1}^r(1-\lz_j^2).$$
For $(r\xx s)$-matrices, we have
$$\lD(z,\lz)=\det(I_r-z\lz^*)=\det(I_{s}-\lz^*z).$$
There exists a scale of weighted Bergman spaces $H^2_s(D),$ for a scalar parameter $s>1+a(r-1)+b,$ with reproducing kernel
$$\KL_s(z,\lz)=\lD(z,\lz)^{-s}.$$
For $s=2+a(r-1)+b$ we obtain the standard Bergman space. The well-known {\bf Faraut-Kor\'anyi binomial formula} \cite{FK1}
$$\lD(z,\lz)^{-s}=\S_\ll(s)_\ll\ \EL^\ll(z,\lz)$$
expresses the kernel functions in terms of the Fischer-Fock reproducing kernels $\EL^\ll.$ Here $(s)_\ll$ denotes the 
{\bf multivariable Pochhammer symbol} (a quotient of Gindikin $\lG$-functions.) Using this formula, one can determine the 
{\bf analytic continuation} of the scale of weighted Bergman spaces as 
$$s>\f a2(r-1)\quad\mbox{(continuous Wallach set)}$$
together with $s=\l\f a2$ for $0\le\l\le r-1$ (discrete Wallach set). A deep result \cite{AZ} says that the corresponding Hilbert space $H_s$ is a Hilbert module if and only if $s$ belongs to the continuous Wallach set. For parameter
$$s=\f dr+(r-\l)\f a2$$
with $1\le\l\le r$ we obtain Hardy type spaces $H^2(\dl_\l D)$ which are supported on the boundary $G$-orbits $\dl_\l D.$ For 
$\l=r$ we obtain the {\bf Shilov boundary} $\dl_r D=S$ and the "standard" Hardy space $H^2(S).$ 

The Hilbert modules mentioned above contain $\PL_E$ as a dense subspace and their reproducing kernel function does not vanish. In this paper we study submodules where this is no longer the case. For any partition $\ll\in\Nl^r_+$ denote by $J^\ll\ic\PL_E$ the ideal generated by $\PL_E^\ll$ (or any linear basis). For example, $J^{0,\;,0}=\PL_E$ and $J^{1,0,\;,0}=\ML_0$ is the maximal ideal at $0\in E.$ These "partition" ideals $J^\ll$ are the main subject of the paper. Define the (partial) {\bf containment ordering} of partitions by 
$$\lm\ge\ll\q\mbox{iff}\q\lm_i\ge\ll_i\ \forall\ 1\le i\le r.$$
This is equivalent to inclusion $[\lm]\ci[\ll]$ of the respective Young diagrams \er{14}. Our first main result is 

\begin{theorem}\label{i} $J^\ll$ has the Peter-Weyl decomposition
$$J^\ll=\S_{\lm\ge\ll}^\op\PL_E^\lm.$$
\end{theorem}

For the matrix space $E=\Cl^{r\xx s}$ ($a=2$) this result is proved in 
\cite[Theorem 4.1]{CEP} using the theory of standard tableaux. Our proof, valid in the more general setting of $J^*$-triples, is based on harmonic analysis of spherical polynomials.

\begin{lemma}\label{j} $J:=\S_{\lm\ge\ll}^\op\PL_E^\lm$ is an ideal in $\PL_E.$
\end{lemma}
\begin{proof} By \cite[Corollary 2.10]{U2} we have for any $\l\in E^*$ and $\lm\in\Nl_+^r$ 
\be{12}\l\.\PL_E^\lm\ic\S_{i=1}^r\PL_E^{\lm+\Le_i}\ee
where the sum is over all $i$ such that $\lm+\Le_i$ is a partition. Since $\lm\ge\ll$ implies $\lm+\Le_i\ge\ll$ it follows that 
$J$ is invariant under multiplication by linear forms and is therefore an ideal in $\PL_E$.
\end{proof}

It follows that $J^\ll\ic J.$ The converse inclusion requires more effort. 

\begin{lemma}\label{k} Let $\ll,\lm,\ln$ be partitions and suppose that 
$$\lS:=\{(pq)^\ll:\ p\in\PL_E^\lm,\ q\in\PL_E^\ln\}$$ 
contains a non-zero vector. Then $\PL_E^\ll$ is spanned by $\lS.$
\end{lemma}
\begin{proof} Since the $\ll$-projection satisfies
$$(pq)^\ll\oc k=((pq)\oc k)^\ll=((p\oc k)(q\oc k))^\ll$$
for all $k\in K,$ the set $\lS$ is $K$-invariant. Hence its linear span $\<\lS\>$ is a $K$-invariant subspace of 
$\PL_E^\ll$ which by assumption is non-zero. Irreducibility implies $\PL_E^\ll=\<\lS\>.$
\end{proof}

The crucial technical result is the following:
 
\begin{lemma}\label{l} Suppose that $\lm$ and $\lm+\Le_i$ are partitions. Then $\PL_E^{\lm+\Le_i}$ is spanned by terms 
$(\l q)^{\lm+\Le_i},$ where $\l\in E^*$ and $q\in\PL_E^\lm.$
\end{lemma}
\begin{proof} By Lemma \ref{k} it suffices to show that some term $(\l q)^{\lm+\Le_i}$ is non-zero. Suppose first that 
$E$ is of tube type, with unit element  $e.$ Denote by $\lF^\lm\in\PL_E^\lm$ the {\bf spherical polynomial} \cite{FK2}. The so-called Pieri formula \cite{K,AZ,St} is
\be{15}(z|e)\lF^\lm(z)=\S_{i}\lF^{\lm+\Le_i}(z)\P_{j\ne i}\f{\lm_i'-\lm_j'+\f a2}{\lm_i'-\lm_j'}\ee
where we define 
$$\lm_i':=\lm_i-\f a2(i-1)$$
and the sum is over all $1\le i\le r$ such that $\lm+\Le_i$ is a partition. We claim that the coefficient 
$(\lm_i'-\lm_j'+\f a2)/(\lm_i'-\lm_j')$ in \er{15} is always $>0.$ If $i<j$ then $\lm_i\ge \lm_j$ and
$$\lm_i'-\lm_j'=\lm_i-\lm_j+\f a2(j-i)\ge\f a2(j-i)\ge\f a2.$$
Hence $\lm_i'-\lm_j'+\f a2\ge a.$ If $j<i$ then $\lm_j\ge\lm_i$ and
$$\lm_j'-\lm_i'=\lm_j-\lm_i+\f a2(i-j)\ge\f a2.$$
Moreover, $\lm_j'-\lm_i'-\f a2=\lm_j-\lm_i+\f a2(i-j-1).$ Since both summands are non-negative, we have $\lm_j'-\lm_i'-\f a2=0$ only if $\lm_j=\lm_i$ and $i-1=j.$ Thus $\lm_{i-1}=\lm_j=\lm_i$ and $\lm+\Le_i$ cannot be a partition. This proves the claim. It follows that $((z|e)\lF^\lm)^{\lm+\Le_i}\ne 0$ whenever $\lm+\Le_i$ is also a partition. 

In the general case, choose a maximal tripotent $e\in E.$ The Fischer-Fock kernel $\EL_e^\lm(z)=\EL^\lm(z,e)$ has a restriction to $E_e$ which is proportional to $\lF^\lm$ by a strictly positive factor. As a consequence, we have
$$(z|e)\EL_e^\lm(z)=\S_i \la_i\ \EL_e^{\lm+\Le_i}(z)$$
with $\la_i>0$ whenever $\lm+\Le_i$ is a partition. This shows $((z|e)\EL_e^\lm)^{\lm+\Le_i}\ne 0.$ 
\end{proof}

\begin{lemma}\label{m} Let $\lm>\ll$ be partitions. Then there exists a partition $\ln\ge\ll$ such that 
$\lm=\ln+\Le_j$ for some $j\le r.$
\end{lemma}
\begin{proof} We have $\lm_i>\ll_i$ for some $i\le r.$ Put $j:=\max\{i\le r:\ \lm_i>\ll_i\}.$ Then 
$$\lm_1\ge\;\ge\lm_j>\ll_j\ge\ll_{j+1}=\lm_{j+1}\ge\ll_{j+2}=\lm_{j+2}\ge\;\ge\ll_r=\lm_r.$$
It follows that $\ln:=\lm-\Le_j$ is a partition with $\ln\ge\ll.$ 
\end{proof}

The {\bf proof of Theorem \ref{i}} is now completed by showing that $J\ic J^\ll,$ i.e., $\PL_E^\lm\ic J^\ll$ for all partitions 
$\lm\ge\ll.$ We use induction over $|\lm|.$ If $\lm=\ll,$ the assertion is trivial. If $\lm>\ll$ there exists a partition 
$\ln\ge\ll$ such that $\PL_E^\lm$ is spanned by terms $(\l q)^\lm,$ with $\l\in E^*$ and $q\in\PL_E^\ln.$ Since 
$|\ln|=|\lm|-1,$ the induction hypothesis implies $\PL_E^\ln\ic J^\ll.$ Applying \er{10} to the character $\lc_\lm$ of 
$\PL_E^\lm$ yields
$$(\l q)^\lm=\I_K dk\ \lc_\lm(k)\ (\l q)\oc k=\I_K dk\ \lc_\lm(k)\ (\l\oc k)(q\oc k).$$
Since $q\oc k\in\PL_E^\ln\ic J^\ll$ for all $k\in K$ and $J^\ll$ is an ideal, it follows that $(\l q)^\lm\in J^\ll$ (the integral is actually performed in a finite-dimensional subspace of $\PL_E$). Therefore $\PL_E^\lm\ic J^\ll.$ This completes the induction step and proves $J\ic J^\ll.$ 

\begin{corollary}\label{T} Let $\ll,\lm$ be partitions. Then $J^\lm\ic J^\ll$ if and only if $\lm\ge\ll.$ 
\end{corollary}
\begin{proof} If $\lm\ge\ll$ then for any partition $\ln\ge\lm$ we have $\ln\ge\ll$ and hence $\PL_E^\ln\ic J^\ll$ by Theorem \ref{i}. Since $\ln$ is arbitrary, it follows that $J^\lm\ic J^\ll.$ Conversely, if $J^\lm\ic J^\ll$ then $\PL_E^\lm\ic J^\ll$ and hence $\lm\ge\ll.$
\end{proof}

As a first application of Theorem \ref{i} we show that for any partition $\ll$ the ideal $J^\ll$ can be written as an intersection of ideals defined by "rectangular" partitions. For $n\in\Nl,$ put
$$n^{(m)}=(n,\;,n,0,\;,0),$$
with $n$ repeated $m$ times. Thus the Young diagram $[n^{(m)}]=[1,m]\xx[1,n].$ Any partition can be written in the form
\be{16}\ll=(n_1^{(\l_1)},n_2^{(\l_2-\l_1)},\;,n_t^{(\l_t-\l_{t-1})},0^{(r-\l_t)}),\ee
where $1\le \l_1<\;<\l_t\le r,\q n_1>n_2>\;>n_t>0.$ In other words,
$$\ll_1=\;=\ll_{\l_1}=n_1>\ll_{1+\l_1}=\;=\ll_{\l_2}=n_2>\;$$

Thus we have $n_s=\ll_j$ for $\l_{s-1}<j\le\l_s.$ In particular, $n_s=\ll_{\l_s}.$ The Young diagram
\be{17}[\ll]=\U_{s=1}^t[n_s^{(\l_s)}]=\U_{s=1}^t[1,\l_s]\xx[1,n_s]\ee
is a (non-disjoint) union of rectangular diagrams.

\begin{proposition}\label{n} Writing $\ll$ in the form \er{16} for "rectangular" partitions $n_s^{(\l_s)}$ we have
$$J^\ll=\C_{s=1}^t J^{n_s^{(\l_s)}}.$$
\end{proposition}
\begin{proof} This follows from Corollary \ref{T} and \er{17}. 
\end{proof}

\begin{proposition}\label{u} The "maximal fibre" at $\lz=0$ is given by 
$$\y H_0=\PL_E^\ll.$$
\end{proposition}
\begin{proof} We first show the easy inclusion $\PL_E^\ll\ic\y H_0.$ Let $p\in\PL_E^\ll$ and $\lm\ge\ll.$ Then for all 
$q\in\PL_E^\lm$ and $\l\in E^*$ we have
$$(T_\l^*p|q)_H=(p|\l q)_H=\S_i(p|(\l q)^{\lm+\Le_i})_H=0$$
since $\lm+\Le_i>\lm\ge\ll$ is different from $\ll.$ Since $T_\l^*p\in H$ and $\lm\ge\ll$ is arbitrary, it follows that 
$T_\l^*p=0.$ 

For the converse, let $f=\S_{\lm\ge\ll}f^\lm\in\y H_0.$ For any partition $\lm>\ll$ there exists a partition 
$\ln\ge\ll$ such that $\PL_E^\lm$ is spanned by terms $(\l q)^\lm,$ where $\l\in E^*$ and $q\in\PL_E^\ln.$ Then
$$(f^\lm|(\l q)^\lm)_H=(f^\lm|\l q)_H=(T_\l^*f^\lm|q)_H=0$$
since $q\in J^\ll$ and $f^\lm\in\y H_0$ by Lemma \ref{f}. It follows that $f^\lm$ is orthogonal to $\PL_E^\lm$ and hence vanishes. Therefore $f=f^\ll\in\PL_E^\ll.$
\end{proof}

\section{Normal projections} 
An irreducible Jordan algebra $E$ with unit element $e$ has a unique {\bf determinant polynomial} $\lD_e:E\to\Cl$ normalized by 
$\lD_e(e)=1$ \cite{FK2,N}. For the matrix algebra $E=\Cl^{r\xx r}$ and the symmetric matrices $E=\Cl_{\x{sym}}^{r\xx r}$ this is the usual determinant. For the antisymmetric matrices $E=\Cl_{\x{asym}}^{2r\xx 2r}$ we obtain the Pfaffian determinant instead. For the spin factor (of rank 2) we have $\lD_e(z)=z\.z=\S_i z_i^2.$ 

The determinant polynomial $\lD_e$ has the semi-invariance property
\be{18}\lD_e(kz)=\lD_e(ke)\lD_e(z)\ee  
for all $k\in K$ and $z\in E.$ The map $\lc:K\to\Tl$ defined by
$$\lc(k):=\lD_e(ke)$$
is a character of $K.$ It follows that for any $k\in K$
$$\lD_{ke}(z):=\lD_e(k^{-1}z)$$
is a Jordan determinant normalized at $ke.$ 

Now consider a frame $e_1,\;,e_r$ and for $1\le m\le r$ put $[m]:=\{1,\;,m\}$ and
$$e_{[m]}:=e_1+\;+e_m\in S_m.$$ 
Let $\lD_{[m]}$ denote the Jordan determinant of the Peirce 2-space $E_{[m]}=E_{e_{[m]}}$ and define the $m$-th 
{\bf Jordan theoretic minor} $N_m$ by
\be{34}N_m(z):=\lD_{[m]}(P_{[m]}z)\ee
where $P_{[m]}$ is the Peirce projection onto $E_{[m]}.$ For any partition $\ll\in\Nl_+^r$ define the 
{\bf conical polynomial}
$$N^\ll:=N_1^{\ll_1-\ll_2}\.N_2^{\ll_2-\ll_3}\:N_r^{\ll_r}.$$
The irreducible $K$-module $\PL_E^\ll$ is the linear span of such polynomials (for various frames) since the highest weight vector is of this form \cite{U1}.  

As a crucial step towards identifying the eigenbundle of the ideals $J^\ll$ (or a Hilbert completion $\o J^\ll$) we consider certain projection mappings. Let $0\le\l\le r$ and consider a rank $\l$ tripotent $c\in S_\l.$ Its Peirce decomposition will be denoted by 
\be{25}E=E_c^2\op E_c^1\op E_c^0=U\op V\op W.\ee
In the matrix case $E=\Cl^{r\xx s},$ with $c=\bb{1_\l}000,$ this corresponds to the decomposition 
\be{26}z=\bb u{v_0}{v^0}w\in\bb UVVW\ee
of $z\in E$ as a block-matrix with $u$ of size $\l\xx\l.$ The Kepler manifold $\c E_\l$ has the tangent space 
$$T_c(\c E_\l)=E_c^2\op E_c^1=U\op V$$
and hence the normal space 
$$T_c^\perp(\c E_\l)=E_c^0=W.$$
Define the {\bf normal projection}
\be{20}\lp_c:\PL_E\to\PL_W,\q\lp_c f(w):=f(c+w)\ee
for $f\in\PL_E$ and $w\in W.$ If $\l=0$ then $\lp_0$ is the identity map.

\begin{lemma}\label{q} Let $\ll\in\Nl_+^r$ and $f\in\PL_E^\ll.$ Then $f|_W=0$ if $\ll_{r-\l+1}>0,$ and 
$f|_W\in\PL_W^{\ll_1,\;,\ll_{r-\l}}$ if $\ll_{r-\l+1}=0.$
\end{lemma}
\begin{proof} We may assume that $f(z)=\EL^\ll(z,b)$ for some $b\in E.$ If $w\in W$ then
$$f(w)=\EL^\ll(w,b)=\EL^\ll(P_W w,b)=\EL^\ll(w,P_W b),$$
where $P_W$ is the Peirce 0-projection onto $W=E^0_c.$ If $\ll_{r-\l+1}>0$ then $\EL^\ll(w,P_W b)=0$ since $W$ has rank 
$r-\l.$ If $\ll_{r-\l+1}=0$ then $\ll=(\ll_1,\;,\ll_{r-\l},0^{(\l)})\equiv(\ll_1,\;,\ll_{r-\l})\in\Nl_+^{r-\l}$ and 
$\EL^\ll(w,P_W b)\in\PL_W^{\ll_1,\;,\ll_{r-\l}}.$
\end{proof}

Since $W=E^c$ is an irreducible $J^*$-triple of rank $r-\l,$ the polynomial algebra $\PL_W$ has its own Peter-Weyl decomposition 
\be{42}\PL_W=\S_{\la\in\Nl_+^{r-\l}}\PL_W^\la\ee 
with respect to the Jordan automorphism group $K_W$ of $W,$ ranging over all partitions 
$$\la=(\la_{\l+1},\;,\la_r)\in\Nl_+^{r-\l}$$ 
of length $r-\l.$ We can therefore consider the ideal 
$$J_W^{\ll^*}:=\S_{\la\in\Nl_+^{r-\l},\ \la\ge\ll^*}\PL_W^\la$$
generated by the "truncated" partition 
\be{22}\ll^*:=(\ll_{\l+1},\;,\ll_r)\ee 
of length $r-\l.$ The main result of this section is

\begin{theorem}\label{r} Let $\ll\in\Nl_+^r$ be a partition. Then for any tripotent $c\in S_\l$ with Peirce 0-space $W$ the 
normal projection \er{20} satisfies
$$J^\ll\xrightarrow{\lp_c}J_W^{\ll^*}.$$ 
In other words, if $f$ has only $K$-components for partitions $\ge\ll,$ then $\lp_c f$ has only $K_W$-components for partitions $\ge\ll^*.$
\end{theorem}
\begin{proof} In the {\bf first step} assume that (i) $c$ is a minimal tripotent ($\l=1$) and (ii) $\ll=n^{(m)}$ has a rectangular Young diagram. In this case $\PL_E^\ll$ is spanned by polynomials of the form 
$$f(z)=N_e(z)^n=\lD_e(P_e z)^n$$ 
where $e\in E$ is a tripotent of rank $m$ and $\lD_e$ is the Jordan determinant of the Peirce 2-space $E_e,$ with Peirce 2-projection $P_e:E\to E_e.$ Regarding $E_e$ as a Jordan algebra with unit element $e$ it has been shown in \cite[Theorem 1]{N} that for a minimal projection $e_1\in E_e$
$$\lD_e(\lx e_1+u)=\lD_e(u)+\lx\lD_{e-e_1}(P_{e-e_1}u)$$
for all $u\in E_e$ and $\lx\in\Cl.$ Since $c$ is a minimal tripotent, $P_ec$ has rank $\le 1$ and hence there exists a minimal tripotent $c_1\in E_e$ such that $P_e c=\lx c_1$ where $\lx\in\Cl.$ Choose $k\in K$ commuting with $P_e$ such that $ke_1=c_1.$ Then $\lk:=k|_{E_e}\in K_{E_e}.$ For $z\in E$ we have
$$P_e(c+z)=\lx c_1+P_ez=\lk(\lx e_1+\lk^{-1}P_ez)=\lk(\lx e_1+P_e k^{-1}z)$$
and semi-invariance of $\lD_e$ implies
$$\f1{\lD_e(ke)}N_e(c+z)=\f1{\lD_e(ke)}\lD_e(P_e(c+z))=\f1{\lD_e(ke)}\lD_e(\lk(\lx e_1+P_e k^{-1}z))$$
$$=\lD_e(\lx e_1+P_e k^{-1}z)=\lD_e(P_e k^{-1}z)+\lx\lD_{e-e_1}(P_{e-e_1}P_e k^{-1}z)$$
$$=\lD_e(P_e k^{-1}z)+\lx\lD_{e-e_1}(P_{e-e_1}k^{-1}z)=N_e(k^{-1}z)+\lx\ N_{e-e_1}(k^{-1}z)$$
since $P_{e-e_1}P_e=P_{e-e_1}.$ Taking the $n$-th power it follows that $f(c+z)=N_e(c+z)^n$ is a linear combination of polynomials $(N_e^{n-h}\ N_{e-e_1}^h)\oc k^{-1},$ where $0\le h\le n.$ These have signature 
$(n^{(m-1)},h,0^{(r-m)}).$ By Lemma \ref{q} the restriction to $W=E^c,$ if not zero, has signature  
$$(n^{(m-1)},h,0^{(r-m-1)})\ge(n^{(m-1)},0,0^{(r-m-1)})=(n^{(m-1)},0^{(r-m)})=(\ll_2,\;,\ll_r).$$
It follows that $\lp_c f\in J_W^{(\ll_2,\;,\ll_r)}$ for all $f\in\PL_E^\ll$ and, a fortiori, for all $f\in J^\ll.$ Thus the assertion holds for minimal tripotents $c$ and rectangular partitions $\ll.$ We can also write the conclusion in the form
\be{70}J^{n^{(m)}}\xrightarrow{\lp_c}J_W^{n^{(m-1)}}.\ee

In the {\bf second step} assume only that $c$ is a minimal tripotent but $\ll\in\Nl_+^r$ is arbitrary. Using the representation
\er{17} we have (since $c$ has rank 1)
$$[\ll^*]=[(\ll_2,\;,\ll_r)]=\U_{s=1}^t[n_s^{(\l_s-1)}].$$ 
By step 1, we have $J^{n_s^{(\l_s)}}\xrightarrow{\lp_c}J_W^{n_s^{(\l_s-1)}}$ for each $s\le t.$ Proposition \ref{g} implies
$$J^\ll=\C_{s=1}^t J^{n_s^{(\l_s)}}\xrightarrow{\lp_c}\C_{s=1}^t J_W^{n_s^{(\l_s-1)}}=J_W^{\ll^*}.$$ 

For the {\bf general case} suppose, by induction, that the assertion is true for tripotents $c'$ of rank $\l-1.$ Let $c$ be a minimal tripotent orthogonal to $c'.$ Then 
$$E^{c+c'}=(E^c)^{c'}$$
and there is a commuting diagram
\be{21}\xymatrix{\PL_E\ar[r]^{\lp_{c}}\ar@/_2pc/[rr]_{\lp_{c+c'}}&\PL_{E^c}\ar[r]^{\lp_{c'}^{E^c}}&\PL_{E^{c+c'}}}\ee
where $\lp_{c'}^{E^c}$ denotes the $c'$-projection relative to $E^c.$ In fact, if $w\in E^{c+c'},$ then 
$c'+w\in E^c$ and
$$\lp_{c'}^{E^c}(\lp_c f)(w)=(\lp_c f)(c'+w)=f(c+(c'+w))=f((c+c')+w)=(\lp_{c+c'}f)(w).$$
By the second part of the proof we have
$$J^\ll\xrightarrow{\lp_c}J_{E^c}^{\ll_2,\;,\ll_r}.$$
The induction hypothesis applied to $E^c$ (of rank $r-1$) and the partition $(\ll_2,\;,\ll_r)$ yields
$$J_{E^c}^{\ll_2,\;,\ll_r}\xrightarrow{\lp_{c'}^{E^c}}J_{(E^c)^{c'}}^{\ll_{\l+1},\;,\ll_r}
=J_{E^{c+c'}}^{\ll_{\l+1},\;,\ll_r}.$$
With \er{21} we obtain $J^\ll\xrightarrow{\lp_{c+c'}}J_{E^{c+c'}}^{\ll_{\l+1},\;,\ll_r}.$ Thus the assertion holds for the tripotent $c+c'$ of rank $\l.$ An induction argument finishes the proof.
\end{proof}

As an {\bf example}, consider the fundamental partition $\ll=1^{(m)}.$ If $c\in S_\l$ with $\l\le m$ then 
${\ll^*}=(\ll_{\l+1},\;\ll_r)=1^{(m-\l)}$ is again a fundamental partition relative to $W=E^c.$ In general, let
$$\ML_X:=\{p\in\PL_E:\ p|_X=0\}=\C_{\lz\in X}\ML_\lz$$
denote the ideal associated with a variety $X\ic E.$ Then Theorem \ref{r} says
$$J^{1^{(m)}}=\ML_{\h E_{m-1}}\xrightarrow{\lp_c}J_W^{1^{(m-\l)}}=\ML_{\h W_{m-\l-1}}.$$
This is clear since $f|_{\h E_{m-1}}=0$ implies $(\lp_c f)(w)=f(c+w)=0$ for $w\in\h W_{m-\l-1},$ because
$$\x{rank}(c+w)=\x{rank}(c)+\x{rank}(w)=\l+\x{rank}(w)\le\l+(m-\l-1)=m-1.$$

\section{The Main Theorem}
Let $c$ be a tripotent of rank $\l,$ with Peirce 0-space $W=E^c.$ We combine the normal projection map $\lp_c$ with the Peter-Weyl projection $\lp_W^{\ll^*}:\PL_W\to\PL_W^{\ll^*}$ onto the lowest $K_W$-type \er{22}. This yields a map 
\be{23}J^\ll\xrightarrow{\lp_c^{\ll^*}}\PL_W^{\ll^*},\q\lp_c^{\ll^*}f:=\lp_W^{\ll^*}(\lp_c f)=(\lp_c f)^{\ll^*}.\ee

\begin{proposition}\label{y} For any tripotent $c\in S_\l$ we have
$$\ML_c J^\ll\ic\ker(\lp_c^{\ll^*}).$$
Thus we have an induced mapping
$$\y J_c^\ll=J^\ll/\ML_c J^\ll\xrightarrow{\lp_c^{\ll^*}}\PL_W^{\ll^*},\q f+\ML_c J^\ll\qi\lp_c^{\ll^*}f.$$
\end{proposition}
\begin{proof} Let $g\in\ML_c$ and $h\in J^\ll.$ By Theorem \ref{r} $\lp_c h$ has only $K_W$-components of type 
$\lm\ge\ll^*.$ Since $g(c)=0$ we have $\lp_c g\in\ML_{W,0}.$ Therefore \er{12} implies that $\lp_c(gh)=(\lp_c g)(\lp_c h)$ has only $K_W$-components of type $\lm>\ll^*.$ Therefore $\lp_c^{\ll^*}(gh)=0.$
\end{proof}

The main result of this paper is 

\begin{theorem}\label{w} For each partition $\ll$ and each tripotent $c,$ with Peirce 0-space $W=E_c^0,$ the map \er{23}, mapping $f\in J^\ll$ to the lowest $K_W$-type of its normal projection $\lp_cf,$ is surjective and has the kernel $\ML_c J^\ll.$ Thus it induces an isomorphism
\be{24}\y J_c^\ll\xrightarrow[\ap]{\lp_c^{\ll^*}}\PL_W^{\ll^*}.\ee
\end{theorem}

In view of Proposition \ref{h} Theorem \ref{w} yields a realization of the full localization bundle $\y J^\ll$ on $E$ as a stratified sum of homogeneous vector bundles supported on the Kepler manifolds $\c E_\l,$ for $0\le \l\le r.$ A description in terms of reproducing kernels will be given in \cite{U3}. The fibre dimension of $\y J^\ll$ on the $\l$-th stratum $\c E_\l$ can be computed using the known dimension formula \cite{U2} for $\dim\PL_E^\ll.$

It is instructive to check the isomorphism \er{24} in the extremal cases $\l=0$ (maximal fibre) and $\l\ge\x{length}(\ll)$ 
(minimal fibre).

\begin{proposition}\label{x} For $\l=0$ we have $c=0,\ W=E^c=E$ and $\ll^*=\ll.$ In this case there is an isomorphism
$$\y J_0^\ll\xrightarrow[\ap]{\lp^\ll}\PL_E^\ll,\quad f+\ML_0 J^\ll\qi f^\ll.$$
\end{proposition} 
\begin{proof} If $\l\in E^*$ and $q\in J^\ll,$ then $q$ has only components of type $\lm\ge\ll$ and $\l q$ has components of type $\lm+\Le_i$ with $\lm\ge\ll.$ It follows that $f\in\ML_0 J^\ll$ has vanishing component $f^\ll=0.$ Thus the map \er{20} is well-defined and surjective since $\PL_E^\ll\ic J^\ll.$ To show injectivity we use an argument analogous to the proof of Theorem \ref{i}. If $f\in J^\ll$ satisfies $f^\ll=0$ then $f$ has only components $f^\lm,$ where $\lm>\ll.$ For each such 
$\lm$ there exists a partition $\ln\ge\ll$ such that $\lm=\ln+\Le_i$ for some $i\le r.$ By Lemma \ref{l}, $\PL_E^\lm$ is spanned by terms $(\l q)^\lm,$ where $\l\in E^*$ and $q\in\PL_E^\ln.$ Since $\l\oc k\in\ML_0$ and $q\oc k\in\PL_E^\ln$ for all 
$k\in K,$ the character integration formula \er{10} yields $f^\lm\in\ML_0\PL_E^\ln\ic\ML_0 J^\ll$ and therefore $f\in\ML_0 J^\ll.$
\end{proof}

In view of Proposition \ref{x} \er{24} is equivalent to
$$\y J_c^\ll\xrightarrow[\ap]{\lp_c}\y{J_W^{\ll^*}}_0.$$
The right hand side corresponds to the maximal fibre of the eigenbundle relative to the normal space $W=E^c.$ This formulation may be valid in more general situations, for stratified varieties which are compatible under passing to the normal space.

At the other extreme, for the minimal (1-dimensional) fibres, suppose that $c$ has maximal rank $r$ or, more generally,
$$\l=\x{rank}(c)\ge\x{length}(\ll).$$ 
Then $\ll^*=(\ll_{\l+1},\;,\ll_r)=0^{(r-\l)}$ and hence $\PL_W^{\ll^*}=\Cl$ (constant functions). This holds even in case 
$\l=r$ where $W=\{0\}.$ Thus \er{24} amounts to an isomorphism
$$\y J_c^\ll\xrightarrow[\ap]{\lp_c^{0^{(r-\l)}}}\PL_W^0\ap\Cl.$$
This follows from Corollary \ref{d} since $c$ is a regular point for $J^\ll.$

A third case that is easily checked is the "Duan-Guo situation" where we have a prime ideal evaluated at a smooth point. In our case this corresponds to fundamental partitions $J^{1^{(\l+1)}}=\ML_{\h E_\l}$ and a tripotent $c\in\c E_\l$ of rank $\l.$ In this case $\ll^*=(1,0,\;,0)$ and hence $\PL_W^{\ll^*}=W^*$ (linear dual space). Thus \er{24} amounts to 
$$\y J_c^{1^{(\l+1)}}\xrightarrow[\ap]{\lp_c^{1,0,\;,0}}\PL_W^{1,0,\;,0}=W^*.$$
Since $W=E^c$ is the normal space to the variety $\h E_\l$ at the smooth point $c,$ this is exactly the result proved 
(in a general context) in \cite{DG}. Note that Theorem \ref{w} includes the singular points of $\h E_\l$ having rank $<\l.$ The corresponding fibres involve non-linear functions on $W.$

The {\bf proof of Theorem \ref{w}} will occupy the rest of this section. Fix a frame $e_1,\;,e_r$ of minimal orthogonal tripotents $e_i.$ Let $1\le\l\le r$ and put 
$$c=e_{[\l]}=e_1+\;+e_\l,$$
noting that the case $c=0$ is covered by Proposition \ref{x}. As a replacement of the usual matrix coordinates, any irreducible $J^*$-triple $E$ has a {\bf joint Peirce decomposition}
$$E=\S_{0\le i\le j\le r}E_{ij}$$
into subspaces $E_{ij}=E_{ji}$ satisfying the "Peirce multiplication rules" \cite{L}
$$\{E_{ij}E_{jm}^*E_{mn}\}\ic E_{in}.$$
Moreover, such a triple product vanishes if there is no possible "matching" of the indices. 

For any $1\le m\le r$ we put $[m]:=\{1,\;,m\}.$ The Peirce decomposition \er{25} under $e_{[m]}$ has the form 
$$E_{[m]}=E_{e_{[m]}}=\S_{i,j\in[m]}E_{ij},\q E_{e_{[m]}}^1=\S_{i\in[m],j\notin[m]}E_{ij},\q E^{e_{[m]}}=\S_{i,j\notin[m]}E_{ij}.$$
This notation applies also for domains not of tube type, where the index 0 occurs. By \er{34} the $m$-th minor $N_m$ has the derivative
$$N_m'(z)\lx=\lD_{[m]}'(P_{[m]}z)P_{[m]}\lx$$
for $\lx\in E,$ where $\lD_{[m]}$ is the Jordan determinant of $E_{[m]}$ with unit element $e_{[m]}.$ For 
$A\in\h\kL$ this implies
\be{33}(A^\dl N_m)(z)=N_m'(z)Az=\lD_{[m]}'(P_{[m]}z)P_{[m]}Az\ee

{\bf Step 1} (of the proof) constructs a "good" spanning set of vectors in $\PL_E^\ll.$ For related arguments, cf. \cite{U1}.

\begin{lemma}\label{z} Let $0\le p,q,j\le r$ and $p,q,j$ distinct. Then the following identities hold
for $z\in E_{pq}, x\in E_{pj}$ and $y\in E_{jq}$:
$$z\b e_q^*=e_p\b\{e_pz^*e_q\}^*,$$
$$e_p\b\{yx^*e_p\}^*=x\b y^*=\{e_qy^*x\}\b e_q^*.$$
Here an index $i\ne 0$ if the identity involves $e_i.$
\end{lemma}
\begin{proof} Since $p\ne q$ we have $z\in E_{e_p}^1$ and hence $z=\{e_pe_p^*z\}.$ Moreover, 
$\{e_qe_p^*e_p\}=0=\{ze_q^*e_p\}.$ Therefore the Jordan triple identity implies 
$$z\b e_q^*=\{e_pe_p^*z\}\b e_q^*=\{e_pe_p^*z\}\b e_q^*-z\b\{e_qe_p^*e_p\}^*=[e_p\b e_p^*,\ z\b e_q^*]$$
$$=-[z\b e_q^*,e_p\b e_p^*]=-\{ze_q^*e_p\}\b e_p^*+e_p\b\{e_pz^*e_q\}^*=e_p\b\{e_pz^*e_q\}^*.$$
This yields the first assertion.

Since $p\ne j$ we have $x\in E_{e_p}^1$ and hence $x=\{e_pe_p^*x\}.$ Here $j=0$ is allowed. 
Moreover, $\{e_py^*x\}=\{e_pe_p^*y\}=0$ since $p$ does not match $j$ or $q.$ Therefore the Jordan triple identity implies
$$x\b y^*-e_p\b\{yx^*e_p\}^*=\{xe_p^*e_p\}\b y^*-e_p\b\{yx^*e_p\}^*$$
$$=[x\b e_p^*,e_p\b y^*]=-[e_p\b y^*,x\b e_p^*]=-\{e_py^*x\}\b e_p^*+x\b\{e_pe_p^*y\}^*=0.$$
Similarly, 
$$\{e_qy^*x\}\b e_q^*-x\b y^*=\{e_qy^*x\}\b e_q^*-x\b\{e_qe_q^*y\}^*$$
$$=[e_q\b y^*,x\b e_q^*]=-[x\b e_q^*,e_q\b y^*]=-\{xe_q^*e_q\}\b y^*+e_q\b\{yx^*e_q\}^*=0$$
since $\{e_qe_q^*y\}=y$ and $\{xe_q^*e_q\}=0=\{yx^*e_q\}.$ This yields the second assertion. 
\end{proof}

We first consider the "annihilation operators" in $\h\kL.$

\begin{proposition}\label{A} Let $q\ne 0$ and $p\notin[q].$ Then 
\be{32}(E_{pj}\b E_{jq}^*)^\dl N_m=0\ee
for $0\le j\le r$ and $1\le m\le r.$
\end{proposition}
\begin{proof} Let $x\in E_{pj}$ and $y\in E_{jq}.$ Lemma \ref{z} implies $A:=x\b y^*=z\b e_q^*,$ where 
$z:=\{xy^*e_q\}\in E_{pq}.$ If $p\notin[m]$ then $AE=\{ze_q^*E\}\ic E_{[m]}^\perp$ (since $p$ does not match $q$). Therefore 
\er{32} follows from \er{33}. If $p\in[m]$ then $p\ne 0$ and $p\notin[q]$ means $q<p.$ Therefore $q\in[m]$ and 
$z\in E_{pq}\ic E_{[m]}.$ It follows that $A\in E_{[m]}\b E_{[m]}^*.$ Since
$$\x{tr}_{E_{[m]}}\ A=\x{const}\.(z|e_q)=0$$ 
it follows that $A$ is a commutator sum in $E_{[m]}\b E_{[m]}^*.$ Now \er{32} follows from semi-invariance of 
$\lD_{[m]}.$
\end{proof}

Now consider the "creation operators" in $\h\kL.$

\begin{lemma}\label{B} Let $p\ne0$ and $q\notin[p]$. Then $E_{pj}\b E_{jq}^*$ is spanned by linear transformations in 
$U\b_0 U^*,\ U\b V^*$ and $W\b_0 W^*$ for the Peirce decomposition \er{25} under $c$ (the subscript 0 means vanishing trace).
\end{lemma}
\begin{proof} Let $A=x\b y^*$ with $x\in E_{pj}$ and $y\in E_{jq}.$ Since $p\ne 0$ Lemma \ref{z} implies $A:=x\b y^*=e_p\b z^*,$
where $z=\{yx^*e_p\}\in E_{pq}.$ There are three cases: If both $p,q\in[\l]$ then $e_p\in U$ and $z\in U.$ Therefore 
$A\in U\b U^*.$ If both $p,q\notin[\l]$ then $e_p\in W$ and $z\in W.$ Therefore $A\in W\b W^*.$ In both cases the relation 
$(e_p|z)=0$ implies vanishing trace. If $p\in[\l]$ and $q\notin[\l]$ then $e_p\in U$ and $z\in V.$ Therefore $A\in U\b V^*.$ Finally, if $p\notin[\l]$ and $q\in[\l]$ then $1\le q\le\l<p$ since $p\ne 0.$ Thus $q\in[p],$ showing that this case cannot occur. 
\end{proof}

\begin{lemma}\label{C} The following commutation relations hold:
$$[U\b U^*,U\b V^*]\ic U\b V^*,\q[W\b W^*,U\b V^*]\ic U\b V^*,\q[U\b U^*,W\b W^*]=0.$$
\end{lemma}
\begin{proof} Let $u,u_1,u_2\in U,\ v\in V$ and $w_1,w_2\in W.$ Then
$$[u_1\b u_2^*,u\b v^*]=\{u_1u_2^*u\}\b v^*-u\b\{vu_1^*u_2\}^*$$
with $\{u_1u_2^*u\}\in U$ and $\{vu_1^*u_2\}\in V.$ Similarly,
$$[w_1\b w_2^*,u\b v^*]=\{w_1w_2^*u\}\b v^*-u\b\{vw_1^*w_2\}^*$$
with $\{w_1w_2^*u\}=0$ and $\{vw_1^*w_2\}\in V.$ Finally,
$$[u_1\b u_2^*,w_1\b w_2^*]=\{u_1u_2^*w_1\}\b w_2^*-w_1\b\{w_2u_1^*u_2\}^*$$
with $\{u_1u_2^*w_1\}=0$ and $\{w_2u_1^*u_2\}=0.$
\end{proof}

\begin{proposition}\label{D} $\PL_E^\ll$ is spanned by terms
\be{47}Y_1^\dl\:Y_s^\dl X_1^\dl \:X_t^\dl Z_1^\dl\:Z_r^\dl N^\ll,\ee
where $s,t,r\ge 0$ and $Y_i\in U\b V^*,\ X_j\in U\b_0 U^*,\ Z_k\in W\b_0 W^*.$
\end{proposition}
\begin{proof} There exists a Cartan subalgebra $\hL\ic\h\kL=E\b E^*$ containing
$$\hL_-:=\Cl\<e_j\b e_j^*:\ 1\le j\le r\>$$
such that the root decomposition
$$\h\kL=\hL\op\S_\la\h\kL_\la$$
has $N^\ll$ as the highest weight vector \cite{U1}. Then $\PL_E^\ll$ is spanned by terms $A_1^\dl\:A_n^\dl N^\ll,$ where 
$n\ge 0$ and $A_i\in\h\kL_\la$ for roots $\la>0.$ By \cite[Theorem 1.7]{U1} 
$$\AL:=\S_{p\ne 0,\ q\notin[p]}\S_j E_{pj}\b E_{jq}^*=\S_{\la>0,\ \la|_{\hL_-}\ne 0}\h\kL_\la$$  
and 
$$\DL:=\S_{p=0}^r\S_j E_{pj}\b E_{jp}^*=\hL\op\S_{\la|_{\hL_-}=0}\h\kL_\la.$$
Thus the positive root spaces $\h\kL_\la$ belong to $\AL\op\DL.$ We claim that
$$[\AL,\DL]\ic\AL.$$ 
In fact, let $x\in E_{pj},y\in E_{jq},u\in E_{ab},v\in E_{ba}.$ Then $x\b y^*=e_p\b z^*$ for some $z\in E_{pq}.$ Hence
$$[x\b y^*,u\b v^*]=[e_p\b z^*,u\b v^*]=\{e_pz^*u\}\b v^*-u\b\{ve_p^*z\}^*.$$
If $\{e_pz^*u\}\ne 0$ then $q$ must match $a$ or $b.$ Assume $q=a.$ Then  
$\{e_pz^*u\}\b v^*\in E_{pb}\b E_{bq}^*\ic\AL.$ Similarly, if $\{ve_p^*z\}\ne 0$ then $p$ must match $a$ or $b.$ Assume $p=a.$ Then $u\b\{ve_p^*z\}^*\in E_{pb}\b E_{bq}^*\ic\AL.$ This proves the claim.

Combining \cite[Theorem 1.7 and Lemma 3.4]{U1} one obtains
$$\DL^\dl N^\ll\ic\Cl N^\ll.$$
It follows that $\PL_E^\ll$ is spanned by terms $A_1^\dl\:A_n^\dl N^\ll,$ where $n\ge 0$ and $A_i\in\AL.$ By Lemma \ref{B} $\AL$ is spanned by transformations of the form $X\in U\b U^*,\ Y\in U\b V^*$ and $Z\in W\b W^*.$ Now the required ordering \er{47} follows from Lemma \ref{C} since
$$Z^\dl X^\dl=X^\dl Z^\dl+[Z^\dl,X^\dl]=X^\dl Z^\dl+[Z,X]^\dl=X^\dl Z^\dl,$$   
$$Z^\dl Y^\dl=Y^\dl Z^\dl+[Z^\dl,Y^\dl]=Y^\dl Z^\dl+[Z,Y]^\dl,$$   
$$X^\dl Y^\dl=Y^\dl X^\dl+[X^\dl,Y^\dl]=Y^\dl X^\dl+[X,Y]^\dl$$
with $[Z,X]=0,\ [Z,Y]\in U\b V^*$ and $[X,Y]\in U\b V^*.$
\end{proof}

{\bf Step 2} concerns vector fields in $U\b U^*.$ Put
$$\ll':=(\ll_1-\ll_{\l+1},\;,\ll_\l-\ll_{\l+1},0^{(r-\l)})$$
and $\h\ll^*=(\ll_{\l+1}^{(\l)},\ll_{\l+1},\;,\ll_r).$ Then
$$N^\ll=N^{\ll'}\ N^{\h\ll^*}.$$

\begin{lemma}\label{F} Let $X\in U\b U^*$ and $f\in\PL_U.$ Then
$$X^\dl(f\oc P_U)=(X|_U^\dl f)\oc P_U.$$
More generally, for $X_1,\;,X_t\in U\b U^*,$
$$X_1^\dl\:X_t^\dl(f\oc P_U)=(X_1|_U^\dl\:X_t|_U^\dl f)\oc P_U$$
\end{lemma}
\begin{proof} Since $X$ preserves the Peirce decomposition \er{25} it follows that $P_UXz=P_UX P_Uz=X|_UP_Uz.$ Therefore 
$$X^\dl(f\oc P_U)z=(f\oc P_U)'(z)Xz=f'(P_Uz)P_UXz=f'(P_Uz)X|_U P_Uz=(X|_U^\dl f)P_Uz.$$
\end{proof}

Consider the subgroup $\h K_U:=\ <\exp U\b U^*>\ \ic\h K.$ In the matrix case \er{26} $\h K_U$ consists of the transformations
$$k\bb u{v_0}{v^0}w=\bb a001\bb u{v_0}{v^0}w\bb{\la}001=\bb{au\la}{a v_0}{v^0\la}w.$$  

\begin{lemma}\label{G}
$$(J_U^{\ll'}\oc P_U)N^{\h\ll^*}\ic J^\ll.$$
\end{lemma}
\begin{proof} Let $h\in J_U^{\ll'}.$ We may assume that $h\in\PL_U^{\ll'}$ and moreover that $h=N_U^{\ll'}\oc k|_U$ for some 
$k\in\h K_U.$ Then $N_U^{\ll'}\oc P_U=N^{\ll'}$ and $N^{\h\ll^*}\oc k=N^{\h\ll^*}.$ Therefore 
$$(h\oc P_U)N^{\h\ll^*}=((N_U^{\ll'}\oc k|_U)\oc P_U)\ N^{\h\ll^*}=((N_U^{\ll'}\oc P_U)\oc k)\ (N^{\h\ll^*}\oc k)$$
$$=(N^{\ll'}\oc k)\ (N^{\h\ll^*}\oc k)=(N^{\ll'}N^{\h\ll^*})\oc k=N^\ll\oc k\in J^\ll.$$
\end{proof}

\begin{proposition}\label{H} Let $X_1,\;,X_t\in U{\b}_0 U^*.$ Then there exists a constant $C$ such that
$$X_1^\dl\:X_t^\dl N^\ll-C N^\ll\in\ML_c J^\ll.$$
\end{proposition}
\begin{proof} Since $\x{rank}(c)=\l=\x{rank}(U)$ it follows that $c$ is a regular point for $J_U^{\ll'}.$ 
Therefore the maximal ideal $\ML_{U,c}\ic\PL_U$ at $c\in U$ satisfies $J_U^{\ll'}+\ML_{U,c}=\PL_U.$ Thus for every 
$f\in J_U^{\ll'}$ we have
$$f-f(c)N_U^{\ll'}\in J_U^{\ll'}\ui\ML_{U,c}=\ML_{U,c}J_U^{\ll'},$$
using \cite[p. 6]{AM} in the last equation. In particular there exist a constant $C$ and $h\in\ML_{U,c}J_U^{\ll'}$ such that 
$$X_1|_U^\dl\:X_t|_U^\dl N_U^{\ll'}=C\ N_U^{\ll'}+h.$$
Applying Lemma \ref{F} we obtain
$$X_1^\dl\:X_t^\dl N^{\ll'}=X_1^\dl\:X_t^\dl(N_U^{\ll'}\oc P_U)=(X_1|_U^\dl\:X_t|_U^\dl N_U^{\ll'})\oc P_U$$
$$=(C\ N_U^{\ll'}+h)\oc P_U=C\ N^{\ll'}+h\oc P_U.$$
Since $\x{tr}_U X_i=0$ it follows that $X_i$ is a commutator sum in $U\b U^*$ and therefore also in 
$E_{[m]}\b E_{[m]}^*$ for all $m>\l.$ Therefore $X_i^\dl N_m=0.$ It follows that
$$X_1^\dl\:X_t^\dl N^\ll=X_1^\dl\:X_t^\dl(N^{\ll'}N^{\h\ll^*})=(X_1^\dl\:X_t^\dl N^{\ll'})N^{\h\ll^*}$$
$$=(C N^{\ll'}+h\oc P_U)N^{\h\ll^*}=C N^\ll+(h\oc P_U)N^{\h\ll^*}.$$
Since $P_Uc=c,$ we have $\ML_{U,c}\oc P_U\ic\ML_c.$ With Lemma \ref{G} it follows that $(h\oc P_U)N^{\h\ll^*}\in\ML_c J^\ll.$
\end{proof}

{\bf Step 3} concerns vector fields in $U\b V^*$ and is the most difficult. We need an analogue of "Cramer's rule" in a Jordan algebra setting. For any Jordan triple $E$ put 
$$Q_xy:=\f12\{xy^*x\}$$ 
and consider the {\bf Bergman endomorphism} $B(x,y)$ on $E$ given by
\be{43}B(x,y)z=z-\{xy^*z\}+Q_xQ_yz.\ee
In the matrix case $E=\Cl^{r\xx s}$ this has the form
$$B(x,y)z=(1_r-xy^*)z(1_s-y^*x).$$
For a unital $J^*$-triple $E$ with unit element $e$ and involution $Q_ez=z^*$ the linear transformation
$$P_z:=Q_zQ_e$$
is called the {\bf quadratic representation} of $z\in E.$ An element $z\in E$ is called invertible if $P_z$ is invertible. In this case, the inverse $z^{-1}$ is defined by $z^{-1}:=P_z^{-1}z.$ For square matrices $E=\Cl^{r\xx r}$ we have $P_zw=zwz$ and $P_z^{-1}z=z^{-1}zz^{-1}=z^{-1}$ is the usual inverse.

The Bergman endomorphism $B(x,y)$ defined in \er{43} belongs to $\h K$ if $\lD(x,y)\ne 0.$ Since $B(x,y)^*=B(y,x)$ and 
$(x\b y^*)^*=y\b x^*$ it follows that $(Q_xQ_y)^*=Q_yQ_x.$ In particular,
$$P_z^*=(Q_zQ_e)^*=Q_eQ_z=Q_eQ_zQ_eQ_e=Q_{Q_ez}Q_e=Q_{z^*}Q_e=P_{z^*}.$$
Let $(z|\lz)$ denote the $K$-invariant inner product normalized by $(e|e)=r.$ At the unit element $e$ the Jordan determinant 
$\lD_e$ has the derivative
$$\lD_e'(e)u=(u|e).$$
For square matrices, the right hand side is the trace of $u.$ If $z\in E$ has strictly positive real part then $P_z^{1/2}$ belongs to the structure group $\h K.$ It follows that
$$\lD_e\oc P_z^{1/2}=\lD_e(z)\lD_e$$
since $\lD_e(z)\lD_e(u)=\lD_e(P_z^{1/2}e)\lD_e(u)=\lD_e(P_z^{1/2}u)=(\lD_e\oc P_z^{1/2})(u)$ for all $u.$ Taking the derivative at $e$ yields
$$\lD_e'(z)(P_z^{1/2}u)=(\lD_e\oc P_z^{1/2})'(e)u=\lD_e(z)\lD_e'(e)u=\lD_e(z)(u|e).$$
By analytic continuation this implies "Cramer's rule"
$$\lD_e'(z)v=\lD_e(z)(P_z^{-1/2}v|e)=\lD_e(z)(v|P_{z^*}^{-1/2}e)=\lD_e(z)\ (v|z^{-*})$$
whenever $z$ is invertible. 

In the following we fix $1\le\l<n\le r$ and put
$$n^+:=\begin{cases}n+1&n<r\\0&n=r\end{cases}$$
The second case arises only for domains not of tube type. Define
$$E_{[m]n^+}:=\S_{i=1}^m E_{in^+}.$$

\begin{proposition}\label{J} Let $1\le m\le n$ and $y\in E_{[m]n^+}.$ Then 
$$N_m'(z)\{e_{[m]}\{e_{[m]}(P_{[m]}z)^*y\}^*z\}=(z|y)\ N_m(z).$$
\end{proposition}
\begin{proof} For $x\in E_{[m]}$ the Peirce multiplication rules imply $\{xy^*e_{[m]}\}=0.$ Hence the Jordan triple identity yields
$$y\b x^*=\{ye_{[m]}^*e_{[m]}\}\b x^*-e_{[m]}\b\{xy^*e_{[m]}\}^*=[y\b e_{[m]}^*,e_{[m]}\b x^*]$$
$$=-[e_{[m]}\b x^*,y\b e_{[m]}^*]=-\{e_{[m]}x^*y\}\b e_{[m]}^*+y\b\{e_{[m]}e_{[m]}^*x\}^*
=-\{e_{[m]}x^*y\}\b e_{[m]}^*+2y\b x^*.$$
Therefore
\be{48}\{e_{[m]}x^*y\}\b e_{[m]}^*=y\b x^*.\ee
Now suppose that $P_{[m]}z$ has maximal rank $m,$ so that $e_{[m]}$ is a supporting tripotent of $P_{[m]}z.$ Let 
$(P_{[m]}z)^{-*}$ denote the inverse of $(P_{[m]}z)^*=Q_{e_{[m]}}P_{[m]}z$ in $E_{[m]}.$ Putting 
$x=(P_{[m]}z)^*$ in \er{48} and evaluating at $(P_{[m]}z)^{-*}$ yields
$$\{\{e_{[m]}(P_{[m]}z)^*y\}e_{[m]}^*(P_{[m]}z)^{-*}\}=\{y(P_{[m]}z)^*(P_{[m]}z)^{-*}\}=\{ye_{[m]}^*e_{[m]}\}=y.$$
By density, we may assume that $z\in E$ has $P_{[m]}z\in E_{[m]}$ of maximal rank $m.$ Applying Cramer's rule to the determinant $\lD_{[m]}$ of $E_{[m]}$ we obtain for $\lx\in E$
$$N_m'(z)\lx=\lD_{[m]}'(P_{[m]}z)P_{[m]}\lx=\lD_{[m]}(P_{[m]}z)\ (P_{[m]}\lx|(P_{[m]}z)^{-*})
=N_m(z)\ (\lx|(P_{[m]}z)^{-*}).$$ 
Putting $\lx=\{e_{[m]}\{e_{[m]}(P_{[m]}z)^*y\}^*z\}$ and using associativity of $(\.|\.)$  we obtain 
$$N_m'(z)\{e_{[m]}\{e_{[m]}(P_{[m]}z)^*y\}^*z\}=N_m(z)\ (\{e_{[m]}\{e_{[m]}(P_{[m]}z)^*y\}^*z\}|(P_{[m]}z)^{-*})$$
$$=N_m(z)\ (z|\{\{e_{[m]}(P_{[m]}z)^*y\}e_{[m]}^*(P_{[m]}z)^{-*}\})=N_m(z)(z|y)$$
\end{proof} 

Let $R_u$ denote the representation 
$$R_uv:=\{uc^*v\}$$ 
of $u\in U=E_{[\l]}$ acting on $v\in E_{[\l]n^+}.$ For matrices $E=\Cl^{r\xx s}$ these are the transformations
$$R_{\bb u000}\bb0{v_0}{v^0}0=\bb u000\bb0{v_0}{v^0}0+\bb0{v_0}{v^0}0\bb u000=\bb0{uv_0}{v^0u}0.$$

\begin{proposition}\label{M} Let $\l\le n\le r$ and $v\in E_{[\l]n^+}.$ Then we have for $\l<m\le n$
\be{38}(z|R_u^{-*}v)\ N_m(z)=N_m'(z)\{cv^*z\}+\S_{i=\l+1}^n N_m'(z)\{e_i\{e_i(P_{[n]}z)^*(R_u^{-*}v)\}^*z\}.\ee
\end{proposition}
\begin{proof} Let $u=P_Uz=P_{[\l]}z$ and assume that $N_\l(z)=\lD_c(u)\ne 0.$ Then $y=R_u^{-*}v\in E_{[\l]n^+}.$ 
If $z_{ab}\in E_{ab}$ with $a,b\in[m]$ then $\{cz_{ab}^*y\}\ne 0$ implies $a,b\in[\l]$ since $a$ or $b$ cannot match $n^+.$ It follows that
$$\{c(P_{[m]}z)^*y\}=\{c(P_{[\l]}z)^*y\}=\{cu^*y\}=\{u^*c^*y\}=R_u^*y=v.$$ 
Now let $\l<i\le m.$ If $z_{ab}\in E_{ab},$ with $a,b\in[n]$ and $\{e_iz_{ab}^*y\}\ne 0,$ then both $a,b$ cannot match $n^+$ and hence one index $b\in[\l].$ Since $i>\l$ it follows that the other index $a=i.$ Thus both indices $a,b\in[m]$ and hence 
\be{39}\{e_i(P_{[n]}z)^*y\}=\{e_i(P_{[m]}z)^*y\}\in E_{in^+}.\ee
Therefore
\be{90}\{e_{[m]}(P_{[m]}z)^*y\}=\{c(P_{[m]}z)^*y\}+\S_{i=\l+1}^m\{e_i(P_{[m]}z)^*y\}
=v+\S_{i=\l+1}^m\{e_i(P_{[n]}z)^*y\}.\ee
If $\l<i\le m$ then $v\in E_{[\l]n^+}\ic E^{e_i}.$ Hence $e_i\b v^*=0$ and
\be{41}e_{[m]}\b v^*=e_{[\l]}\b v^*=c\b v^*.\ee
If $j\in[m]$ with $j\ne i$ then $\{e_i(P_{[n]}z)^*y\}\in E_{in^+}\ic E^{e_j}$ and hence 
$e_j\b\{e_i(P_{[n]}z)^*y\}^*=0.$ Therefore
\be{37}e_{[m]}\b\{e_i(P_{[n]}z)^*y\}^*=e_i\b\{e_i(P_{[n]}z)^*y\}^*.\ee
Applying Proposition \ref{J} together with \er{10},\er{41} and \er{37} we obtain 
$$(z|y)\ N_m(z)=N_m'(z)\{e_{[m]}\{e_{[m]}(P_{[m]}z)^*y\}^*z\}=N_m'(z)\{e_{[m]}(v+\S_{i=\l+1}^m\{e_i(P_{[n]}z)^*y\})^*z\}$$
$$=N_m'(z)\{e_{[m]}v^*z\}+\S_{i=\l+1}^m N_m'(z)\{e_{[m]}\{e_i(P_{[n]}z)^*y\}^*z\}$$
$$=N_m'(z)\{cv^*z\}+\S_{i=\l+1}^m N_m'(z)\{e_i\{e_i(P_{[n]}z)^*y\}^*z\}$$
$$=N_m'(z)\{cv^*z\}+\S_{i=\l+1}^n N_m'(z)\{e_i\{e_i(P_{[n]}z)^*y\}^*z\}.$$
In the last step we use that for $m<i\le n$ we have $\{e_i\{e_i(P_{[n]}z)^*y\}^*z\}\in\S_{k}E_{ik}\ic E_{[m]}^\perp.$ and hence 
$N_m'(z)\{e_i\{e_i(P_{[n]}z)^*y\}^*z\}=0.$
\end{proof}

The transformation $R_u$ on $U$ has determinant $\lD_c(u)^\ld=N_\l(z)^\ld,$ where $u=P_cz$ and
$$\ld:=\dim E_{in^+}=\begin{cases}a&n<r\\b&n=r\end{cases}.$$
Computing $R_u^{-1}$ with (the usual) Cramer's rule it follows that for fixed $v$ there exists a "cofactor polynomial" 
$\lQ^v_\lx(z)$ depending linearly on $\lx\in E$ such that
\be{49}\lQ^v_\lx(z)=N_\l(z)^\ld\ (\lx|R_u^{-*}v)\ee
for all $z$ with $N_\l(z)\ne 0$ (a dense open subset of $E$). We apply this construction to $\lx=z$ and to 
$\lx=\{v_i^\la e_i^*(P_{[n]}z)\},$ where $\l<i\le n$ and $v_i^\la$ is an orthonormal basis of $E_{in^+}.$ 

\begin{lemma}\label{L} The "cofactor polynomials" $\lQ^v_z(z)$ and $\lQ^v_{\{v_i^\la e_i^*(P_{[n]}z)\}}(z)$ vanish at 
$c=e_{[\l]}$ and hence belong to $\ML_c.$
\end{lemma}
\begin{proof} Since $N_\l(c)=1$ and $R_c=\x{id}$ \er{49} implies 
$$\lQ^v_\lx(c)=N_\l(c)^\ld\ (\lx|R_c^{-*}v)=(\lx|v)$$
for $\lx$ evaluated at $z=c.$ For $\lx=z$ evaluating at $z=c$ yields $(c|v)=0$ since $v\in E_{[\l]n^+}\ic U^\perp.$ For 
$\lx=\{v_i^\la e_i^*(P_{[n]}z)\},$ with $\l<i\le n,$ evaluating at $z=c$ yields $\lx=\{v_i^\la e_i^*(P_{[n]}c)\}
=\{v_i^\la e_i^*c\}=0$ since $i>\l.$
\end{proof}

\begin{lemma}\label{I} If $Y\in U\b V^*$ then $Yc=0.$
\end{lemma}
\begin{proof} Let $Y=u\b v^*$ with $u\in U$ and $v\in V.$ The Peirce multiplication rules imply 
$Yc=\{uv^*c\}\in\{UV^*U\}=\{0\}.$
\end{proof}

\begin{proposition}\label{P} Let $Y\in U\b V^*.$ Then $Y^\dl N^\ll\in\ML_c J^\ll.$
\end{proposition}
\begin{proof} By Lemma \ref{z} we may assume that $Y=c\b v^*$ for some $v\in E_{[\l]n^+}$ with $\l\le n\le r.$ Multiplying the identity \er{38} by $N_\l(z)^\ld$ and applying the definition of $\lQ^v_z(z)$ and $\lQ^v_{\{v_i^\la e_i^*(P_{[n]}z)\}}(z)$ one obtains 
\be{44}N_\l(z)^\ld\ (c\b v^*)^\dl N_m
=\lQ^v_z(z)\ N_m(z)-\S_{i=\l+1}^n\S_\la\lQ^v_{\{v_i^\la e_i^*(P_{[n]}z)\}}(z)\ (e_i\b v_i^{\la*})^\dl N_m\ee
whenever $\l<m\le n.$ In case $n<r$ we have $(c\b v^*)^\dl N_m=0$ if $\l<n<m\le r$ since $n^+\le m$ and hence 
$c,v\in E_{[m]}$ with $(c|v)=0$ (for $n=r$ the condition is empty). Therefore \er{44} implies
$$N_\l^\ld\ \f{(c\b v^*)^\dl N^{\h\ll^*}}{N^{\h\ll^*}}
=N_\l^\ld\ \S_{m=\l+1}^r(\ll_m-\ll_{m+1})\f{(c\b v^*)^\dl N_m}{N_m}
=N_\l^\ld\ \S_{m=\l+1}^n(\ll_m-\ll_{m+1})\f{(c\b v^*)^\dl N_m}{N_m}$$
$$=\S_{m=\l+1}^n(\ll_m-\ll_{m+1})\(\lQ^v_z
-\S_{i=\l+1}^n\S_\la\lQ^v_{\{v_i^\la e_i^*(P_{[n]}z)\}}\f{(e_i\b v_i^{\la*})^\dl N_m}{N_m}\)$$
$$=C\.\lQ^v_z-\S_{i=\l+1}^n\S_\la\lQ^v_{\{v_i^\la e_i^*(P_{[n]}z)\}}\S_{m=\l+1}^n(\ll_m-\ll_{m+1})\f{(e_i\b v_i^{\la*})^\dl N_m}{N_m}.$$
Here the constant 
$$C=\S_{m=\l+1}^n(\ll_m-\ll_{m+1})=\begin{cases}\ll_{\l+1}-\ll_{n^+}&n<r\\\ll_{\l+1}&n=r\end{cases}.$$ 
On the other hand, if $\l<i\le n$ then
$$\f{(e_i\b v_i^{\la*})^\dl N^{\h\ll^*}}{N^{\h\ll^*}}=\S_{m=\l+1}^r(\ll_m-\ll_{m+1})\f{(e_i\b v_i^{\la*})^\dl N_m}{N_m}
=\S_{m=\l+1}^n(\ll_m-\ll_{m+1})\f{(e_i\b v_i^{\la*})^\dl N_m}{N_m}$$
since $(e_i\b v_i^{\la*})^\dl N_m=0$ if $i\le n<m.$ It follows that
$$N_\l(z)^\ld\ (c\b v^*)^\dl N^{\h\ll^*}=C\lQ^v_z(z)\ N^{\h\ll^*}
-\S_{i=\l+1}^n\S_\la\lQ^v_{\{v_i^\la e_i^*(P_{[n]}z)\}}(e_i\b v_i^{\la*})^\dl N^{\h\ll^*}$$
and therefore
\be{46}N_\l^\ld N^{2\ll'}(Y^\dl N^{\h\ll^*})=N^{2\ll'}\(C\lQ^v_z\ N^{\h\ll^*}-\S_{i=\l+1}^n
\S_\la\lQ^v_{\{v_i^\la e_i^*(P_{[n]}z)\}}(e_i\b v_i^{\la*})^\dl N^{\h\ll^*}\)$$
$$=C\.N^{\ll'}\lQ^v_z\ N^\ll-\S_{i=\l+1}^n\S_\la\lQ^v_{\{v_i^\la e_i^*(P_{[n]}z)\}}N^{2\ll'}(e_i\b v_i^{\la*})^\dl N^{\h\ll^*}.\ee
For any $A\in\h\kL$ we have $A^\dl N^\ll=A^\dl(N^{\ll'}N^{\h\ll^*})=(A^\dl N^{\ll'})N^{\h\ll^*}+N^{\ll'}(A^\dl N^{\h\ll^*})$
and therefore
$$N^{2\ll'}(A^\dl N^{\h\ll^*})=N^{\ll'}(A^\dl N^\ll)-(A^\dl N^{\ll'})N^\ll\in J^\ll.$$ 
This implies 
$$N^{2\ll'}(e_i\b v_i^{\la*})^\dl N^{\h\ll^*}\in J^\ll$$ 
and \er{46} and Lemma \ref{L} imply $N_\l^\ld N^{2\ll'}(Y^\dl N^{\h\ll^*})\in\ML_c J^\ll.$ Since $N_\l(c)=1$ it follows that
$$N^{2\ll'}(Y^\dl N^{\h\ll^*})\in\ML_c J^\ll.$$
We have $Y^\dl N^{\ll'}\in\ML_c$ since $Yc=0$ by Lemma \ref{I}. Therefore the identity
$$N^{\ll'}(Y^\dl N^\ll)=N^{\ll'}\((Y^\dl N^{\ll'})N^{\h\ll^*}+N^{\ll'}(Y^\dl N^{\h\ll^*})\)
=(Y^\dl N^{\ll'})N^\ll+N^{2\ll'}(Y^\dl N^{\h\ll^*})$$
shows that $N^{\ll'}(Y^\dl N^\ll)\in\ML_c J^\ll.$ Since $N^{\ll'}(c)=1$ this completes the proof.
\end{proof}

{\bf Step 4} constructs an equivariant cross-section to the (surjective) map \er{23}. We need a "compression formula" for Jordan determinants which in the matrix case amounts to the well-known relation
$$\det\bb u{v_0}{v^0}w=\det(u)\det(w-v^0u^{-1}v_0)$$
for block-matrices, with the $(\l\xx\l)$-submatrix $u$ invertible. If $E$ is of tube type (i.e. $r=s$ in the matrix case) its Jordan algebra determinant $N(z)$ satisfies
\be{27}N(B(x,y)z)=\lD(x,y)^2\ N(z).\ee
This follows from the fact that $B(x,y)$ belongs to $\h K$ whenever $\lD(x,y)\ne 0.$ Using the decomposition \er{25} define the open dense subset
$$\lO:=\{z=u+v+w\in E:\ u\in U\ \mbox{invertible}\}$$
and a rational map $\lo:\lO\to W$ by
$$\lo(z):=w-Q_vu^{-*}=w-\f12\{vu^{-1}v\}.$$
Here $u^{-1}$ denotes the inverse of $u\in U$ and $u^{-*}\in U$ denotes the inverse of $u^*$ for the involution 
$u^*=Q_cu$ on $U.$ Consider the subgroup $\h K_W:=\ <\exp W\b W^*>$ of $\h K.$ In the matrix case \er{26} $\h K_W$ consists of the transformations
$$k\bb u{v_0}{v^0}w=\bb 100d\bb u{v_0}{v^0}w\bb100{\ld}=\bb u{v_0\ld}{dv^0}{dw\ld}.$$ 
Any $k\in\h K_W$ satisfies $k(u+v+w)=u+kv+kw,$ with $kv\in V$ and $kw\in W.$ Therefore $\h K_W$ leaves $\lO$ invariant and
$$\lo(kz)=kw-Q_{kv}u^{-*}=k(w-Q_v u^{-*})=k|_W\lo(z).$$

\begin{lemma}\label{S} Let $z=u+v+w\in E$ with $u$ invertible. Then we have the "compression formula"
\be{29}N_m(z)=\lD_c(u)\ N_{m-\l}^W(\lo(z))\ee
for $\l<m\le r.$ Here 
$$N_{m-\l}^W=N_{e_{[m]}-c}^W=N_{e_{\l+1}+\;+e_m}^W$$ 
denotes the $(m-\l)$-th minor relative to $W.$
\end{lemma}
\begin{proof} The Peirce multiplication rules imply $Q_vu^{-*}\in W.$ Therefore
$$B(v,u^{-*})(u+v+w)=u+(w-Q_vu^{-*})=u+\lo(z)$$
has no $V$-component. Moreover, $\lD(v,u^{-*})=1.$ Applying \er{27} to the unital $J^*$-triple $E_{[m]}$ we obtain
$$N_m(z)=N_m(u+v+w)=\lD(v,u^{-*})^2\ N_m(u+v+w)$$
$$=N_m(B(v,u^{-*})(u+v+w))=N_m(u+\lo(z))=\lD_c(u)\ N_{m-\l}^W(\lo(z)).$$ 
\end{proof}

For $\la=(\la_{\l+1},\;,\la_r)\in\Nl_+^{r-\l}$ put 
$$\h\la:=(\la_{\l+1}^{(\l)},\la_{\l+1},\;,\la_r)\in\Nl_+^r$$ 
and denote by
$$N_W^\la=\P_{m=\l+1}^r(N_{m-\l}^W)^{\la_m-\la_{m+1}}$$
the conical polynomial on $W$ relative to $\la.$ If $u=P_Uz$ is invertible, then \er{29} implies
$$N^{\h\la}(z)=\P_{m=\l+1}^r N_m^{\la_m-\la_{m+1}}(z)$$
$$=\P_{m=\l+1}^r\lD_c(u)^{\la_m-\la_{m+1}}(N_{m-\l}^W(\lo(z)))^{\la_m-\la_{m+1}}=\lD_c(u)^{\la_{\l+1}}\ N_W^\la(\lo(z)).$$

\begin{corollary}\label{E} Let $z=u+v+w\in\lO$ and $Z_1,\;,Z_r\in W\b W^*.$ Then
$$(Z_1^\dl\:Z_r^\dl N^{\h\la})(z)=\lD_c(u)^{\la_{\l+1}}(Z_1|_W^\dl\:Z_r|_W^\dl N_W^\la)(\lo(z)).$$
\end{corollary}
\begin{proof} Each $k\in\h K_W$ leaves $\lO$ invariant and satisfies $k|_U=\x{id}.$ Therefore Lemma \ref{S} implies
$$N^{\h\la}(kz)=\lD_c(u)^{\la_{\l+1}}\ N_W^\la(\lo(kz))=\lD_c(u)^{\la_{\l+1}}\ N_W^\la(k|_W\lo(z)).$$
Taking a 1-parameter group $k_t=\exp(tZ)$ with $Z\in W\b W^*$ this implies
$$(Z^\dl N^{\h\la})(z)=\f{d}{dt}|_{t=0}N^{\h\la}(k_tz)=\lD_c(u)^{\la_{\l+1}}\ \f{d}{dt}|_{t=0}N_W^\la(k_t|_W\lo(z))
=\lD_c(u)^{\la_{\l+1}}\ (Z|_W^\dl N_W^\la(\lo(z)).$$
Iterating this relationship, the assertion follows.
\end{proof}

\begin{proposition}\label{U} For each $\ll\in\Nl_+^r$ there is a unique linear map $\Lt_\ll:\PL_W^{\ll^*}\to\PL_E^\ll$ with
$$\Lt_\ll N_W^{\ll^*}=N^\ll,$$
which for all $\lf\in\PL_W^{\ll^*}$ has the cross-section property
\be{30}\lp_c^{\ll^*}(\Lt_\ll\lf)=\lf\ee
and the invariance property
\be{28}(\Lt_\ll\lf)\oc k=\Lt_\ll(\lf\oc k|_W)\ee
under $k\in\h K_W.$ 
\end{proposition}
\begin{proof} Let $u=P_Uz$ be invertible. Any $k\in\h K_W$ satisfies $P_U(kz)=P_Uz=u$ and $\lo(kz)=k|_W\lo(z).$ It follows that
$$N^{\h\la}(kz)=\lD_c(P_U(kz))^{\la_{\l+1}}\ N_W^\la(\lo(kz))=\lD_c(u)^{\la_{\l+1}}N_W^\la(k|_W\lo(z)).$$
Therefore a relation $\S_i C_i\ N_W^\la\oc k_i|_W=0,$ for constants $C_i$ and $k_i\in\h K_W,$ implies 
$$\S_i C_i\ N^{\h\la}(k_iz)=\lD_c(u)^{\la_{\l+1}}\S_i C_i\ N_W^\la(k_i|_W(\lo(z))=0$$ 
on a dense open subset of $E$ and hence on all of $E.$ By irreducibility under $K_W,$ every $\lf\in\PL_W^\la$ has the form 
$\lf=\S_i C_i\ N_W^\la\oc k_i|_W.$ Thus there is a well-defined linear map $\Lt_\la:\PL_W^\la\to\PL_E^{\h\la}$ given by
$$\lf=\S_i C_i\ N_W^\la\oc k_i|_W\qi\Lt_\la\lf:=\S_i C_i\ N^{\h\la}\oc k_i.$$
Then
$$\Lt_\la N_W^\la=N^{\h\la}$$
and we have the invariance property
\be{28}(\Lt_\la\lf)\oc k=\Lt_\la(\lf\oc k|_W)\ee
for all $k\in\h K_W,$ as follows with
$$\lf\oc k|_W=(\S_i C_i\ N_W^\la\oc k_i|_W)\oc k|_W=\S_i C_i\ N_W^\la\oc(k_i|_W k|_W)=\S_i C_i\ N_W^\la\oc(k_ik)|_W$$
and
$$(\Lt_\la\lf)\oc k=\S_i C_i\ (N^{\h\la}\oc k_i)\oc k=\S_i C_i\ N^{\h\la}\oc(k_i k).$$
Since $\lp_c^\la N^{\h\la}=N_W^\la$ this also yields the cross-section property
\be{30}\lp_c^\la(\Lt_\la\lf)=\lf\ee
for all $\lf\in\PL_W^\la.$ Since $N^{\ll'}$ is invariant under $\h K_W$ and $\lp_c N^{\ll'}=1,$ the required cross-section 
$\Lt_\ll$ is defined by
$$\Lt_\ll\lf:=N^{\ll'}(\Lt_{\ll^*}\lf).$$
\end{proof}

{\bf Step 5} completes the proof of Theorem \ref{w}.

\begin{proposition}\label{E} $\PL_E^\ll$ is spanned by terms
$$(Y_1^\dl\:Y_s^\dl X_1^\dl\:X_t^\dl N^\ll)\oc k$$
where $s,t\ge 0,$ $Y_i\in U\b V^*,\ X_j\in U\b U^*$ and $k\in\h K_W.$
\end{proposition}
\begin{proof} By Proposition \ref{D} it suffices to consider terms
$$Y_1^\dl\:Y_s^\dl X_1^\dl\:X_t^\dl Z_1^\dl\:Z_r^\dl N^\ll$$
where $s,t,r\ge 0$ and $Y_i\in U\b V^*,\ X_j\in U\b_0 U^*$ and $Z_k\in W\b_0 W^*.$ Since $\h K_W$ acts irreducibly on 
$\PL_W^{\ll^*}$ it follows that
$$Z_1|_W^\dl\:Z_r|_W^\dl N_W^{\ll^*}\in\ <N_W^{\ll^*}\oc k|_W:\ k\in\h K_W>.$$
The cross-section $\Lt_\ll$ defined in Proposition \ref{U} satisfies $\Lt_\ll N_W^{\ll^*}=N^\ll,\ \Lt_\ll\oc Z|_W^\dl=Z^\dl\oc\Lt_\ll$ and $\Lt_\ll(\lf\oc k|_W)=(\Lt_\ll\lf)\oc k.$ Therefore
$$Z_1^\dl\:Z_r^\dl N^\ll=Z_1^\dl\:Z_r^\dl(\Lt_\ll N_W^{\ll^*})=\Lt_\ll(Z_1|_W^\dl\:Z_r|_W^\dl N_W^{\ll^*})$$ 
is a linear combination of $\Lt_\ll(N_W^{\ll^*}\oc k|_W)=(\Lt_\ll N_W^{\ll^*})\oc k=N^\ll\oc k$ for $k\in\h K_W.$ The proof is concluded by noting that for $k\in\h K_W$ we have
$$X^\dl(f\oc k)=(X^\dl f)\oc k$$
if $X\in U\b U^*$ and
$$Y^\dl(f\oc k)=((kYk^{-1})^\dl f)\oc k$$
if $Y\in U\b V^*,$ with $kYk^{-1}\in U\b V^*.$ This follows from
$$Y^\dl(f\oc k)(z)=(f\oc k)'(z)Yz=f'(kz)kYz=f'(kz)(kYk^{-^1})kz=((kYk^{-1})^\dl f)(kz).$$
\end{proof}
 
\begin{lemma}\label{Q} 
$$\PL_E^\ll\ic\<N^\ll\oc\h K_W\> +\ML_c J^\ll=\Lt_\ll\PL_W^{\ll^*}+\ML_c J^\ll.$$
\end{lemma}
\begin{proof} Since $\ML_c\oc\h K_W=\ML_c$ it suffices by Proposition \ref{E} to show that 
\be{45}g:=Y_1^\dl\:Y_s^\dl X_1^\dl\:X_t^\dl N^\ll\in\<N^\ll\>+\ML_c J^\ll,\ee 
where $s,t\ge 0$ and $Y_i\in U\b V^*,\ X_j\in U\b U^*.$ By Lemma \ref{H} there exist a constant $C$ and $h\in\ML_c J^\ll$ such that $X_1^\dl\:X_t^\dl N^\ll=C N^\ll+h.$ This proves \er{45} if $s=0.$ Now let $s\ge 1.$ Proposition \ref{P} implies $Y^\dl N^\ll\in\ML_c J^\ll$ for all $Y\in U\b V^*.$ With $Y^\dl(pq)=(Y^\dl p)q+p(Y^\dl q)$ we also have 
$$Y^\dl(\ML_c J^\ll)\ic\ML_c J^\ll$$
since $Yc=0$ implies $(Y^\dl p)(c)=p'(c)Yc=0$ for all $p\in\PL_E$ and hence $Y^\dl\PL_E\ic\ML_c.$ Therefore 
$Y^\dl h\in\ML_c J^\ll$ and hence $g=Y_1^\dl\:Y_s^\dl(C N^\ll+h)\in\ML_c J^\ll.$
\end{proof}

\begin{proposition}\label{R} $\x{ker}(\lp_c^{\ll^*})\ic\ML_c J^\ll.$
\end{proposition}
\begin{proof} Let $(\lf_i)_{i\in I}$ be a basis of $\PL_W^{\ll^*}.$ Then $\Lt_\ll\lf_i\in\PL_E^\ll$ satisfies 
$\lp_c^{\ll^*}\Lt_\ll\lf_i=\lf_i$ and $(\Lt_\ll\lf_i)$ are linearly independent. Let $(p_j)_{j\in J}$ be a basis of 
$\ker(\lp_c^{\ll^*})\ui\PL_E^\ll.$ Then $(\Lt_\ll\lf_i)_{i\in I}\iu(p_j)_{j\in J}$ form a basis of $\PL_E^\ll.$ Now let 
$f\in\x{ker}(\lp_c^{\ll^*})\ic J^\ll.$ By definition of $J^\ll$ there exist $f_i\in\ML_c,\ a_i\in\Cl$ and $q_j\in\PL_E$ such that
\be{46}f=\S_{i\in I}(f_i+a_i)(\Lt_\ll\lf_i)+\S_{j\in J}q_jp_j=\S_{i\in I}f_i(\Lt_\ll\lf_i)+\S_{j\in J}q_jp_j
+\S_{i\in I}a_i(\Lt_\ll\lf_i).\ee
By Lemma \ref{Q} each $p_j$ can be written as $p_j=\Lt_\ll\lq_j+h_j,$ where $\lq_j\in\PL_W^{\ll^*}$ and $h_j\in\ML_c J^\ll.$ Since $p_j\in\ker(\lp_c^{\ll^*})$ and $\ML_c J^\ll\ic\ker(\lp_c^{\ll^*})$ by Proposition \ref{y} it follows that
$$0=\lp_c^{\ll^*}p_j=\lp_c^{\ll^*}(\Lt_\ll\lq_j)+\lp_c^{\ll^*}h_j=\lp_c^{\ll^*}(\Lt_\ll\lq_j)=\lq_j.$$
Therefore $p_j=h_j\in\ML_c J^\ll.$ By Theorem \ref{r} $\lp_c\Lt_\ll\lf_i$ has only signatures $\ge\ll^*.$ Since $\lp_c f_i$ vanishes at $0\in W$, it follows that $(\lp_c f_i)(\lp_c\Lt_\ll\lf_i)$ has only signatures $>\ll^*.$ The same holds for $\lp_c p_j$ and hence for $(\lp_cq_j)(\lp_cp_j)$ since $p_j\in\x{ker}(\lp_c^{\ll^*}).$ With $\lp_c^{\ll^*}\Lt_\ll\lf_i=\lf_i$ it follows from \er{46} that
$$0=\lp_c^{\ll^*}f=\S_{i\in I}\((\lp_c f_i)(\lp_c\Lt_\ll\lf_i)\)^{\ll^*}
+\S_{j\in J}\((\lp_c q_j)(\lp_c p_j)\)^{\ll^*}+\S_{i\in I}a_i(\lp_c^{\ll^*}\Lt_\ll\lf_i)=\S_{i\in I}a_i\lf_i.$$
Since $(\lf_i)_{i\in I}$ are linearly independent, we have $a_i=0$ for all $i\in I.$ With \er{46} we obtain $f\in\ML_c J^\ll$ since $\Lt_\ll\lf_i\in J^\ll$ and $p_j\in\ML_c J^\ll.$ 
\end{proof}

\end{document}